\documentclass[arxiv]{agn_article}
\usepackage[french,english]{babel}
\usepackage{csquotes}

\usepackage{amssymb,amsmath,amsthm,agn_bib}

\usepackage{agn_macros,agn_authors}
\usepackage{mathrsfs}
\usepackage{graphicx}
\usepackage{arydshln}
\usepackage{multicol,enumitem}
\usepackage[normalem]{ulem}

\newcounter{alphasect}
\def\alphainsection{0}

\let\oldsection=\section
\def\section{%
  \ifnum\alphainsection=1%
    \addtocounter{alphasect}{1}
  \fi%
\oldsection}%

\renewcommand\thesection{%
  \ifnum\alphainsection=1%
    \Alph{alphasect}%
  \else%
    \arabic{section}%
  \fi%
}%

\newenvironment{alphasection}{%
  \ifnum\alphainsection=1%
    \errhelp={Let other blocks end at the beginning of the next block.}
    \errmessage{Nested Alpha section not allowed}
  \fi%
  \setcounter{alphasect}{0}
  \def\alphainsection{1}
}{%
  \setcounter{alphasect}{0}
  \def\alphainsection{0}
}%

\AtBeginBibliography{\footnotesize}

\newtheorem{theorem}{Theorem}[section]

\newtheorem{lemma}[theorem]{Lemma}
\newtheorem{observation}[theorem]{Observation}
\newtheorem{corollary}[theorem]{Corollary}

\theoremstyle{definition}

\newtheorem{remark}[theorem]{Remark}

\usepackage{xcolor,hyperref}
\definecolor{darkblue}{rgb}{0,0,0.6}
\hypersetup{
    colorlinks=true,       
    linkcolor=darkblue,          
    citecolor=darkblue,        
    filecolor=darkblue,      
    urlcolor=darkblue           
}

\usepackage[colorinlistoftodos,prependcaption,textsize=tiny]{todonotes}
\newcommand{\skalarProd}[2]{\big\langle#1,#2\big\rangle}

\newcommand{\inc}{\boldsymbol{\operatorname{inc}}\,}
\renewcommand{\D}{\operatorname{D}\hspace{-1pt}}
\newcommand{\curl}{\operatorname{curl}}
\newcommand{\Anti}{\operatorname{Anti}}

\newlist{thmenum}{enumerate}{1}
\setlist[thmenum]{label=\upshape(\alph*)}

\usepackage{tikz}

\begin{document}
\begin{tikzpicture}[remember picture, overlay]
 \node [xshift=-1cm,yshift=15cm,rotate=-90] at (current page.south east)
 {Proceedings of the Royal Society of Edinburgh (2021), doi: \href{https://www.doi.org/10.1017/prm.2021.62}{10.1017/prm.2021.62}.
 };
\end{tikzpicture}
\numberwithin{equation}{section}

\numberwithin{equation}{section}
\title{$L^p$-trace-free generalized Korn inequalities for incompatible tensor fields in three space dimensions}
\knownauthors[lewintan]{lewintan,neff}

\maketitle

\begin{center}\textit{In memoriam of Sergio Dain [1970-2016],\\ who gave the first proof of the trace-free Korn's inequality\\ on bounded Lipschitz domains.}\end{center}

\begin{abstract}
For $1<p<\infty$ we prove an $L^p$-version of the generalized trace-free Korn  inequality for incompatible tensor fields $P$ in $ W^{1,\,p}_0(\Curl; \Omega,\R^{3\times3})$. More precisely, let $\Omega\subset\R^3$ be a bounded Lipschitz domain. Then there exists a constant $c>0$ such that
\[
\norm{ P }_{L^p(\Omega,\R^{3\times3})}\leq c\,\left(\norm{\dev \sym P }_{L^p(\Omega,\R^{3\times3})} + \norm{ \dev \Curl P }_{L^p(\Omega,\R^{3\times3})}\right)
\]
holds for all tensor fields $P\in  W^{1,\,p}_0(\Curl; \Omega,\R^{3\times3})$, i.e., for all $P\in W^{1,\,p}(\Curl; \Omega,\R^{3\times3})$ with vanishing tangential trace $ P\times \nu=0 $ on $ \partial\Omega$
where  $\nu$ denotes the outward unit normal vector field to $\partial\Omega$ and  $\dev P \coloneqq P -\frac13 \tr(P) {\cdot}\id$ denotes the deviatoric (trace-free) part of $P$. We also show the norm equivalence{
\begin{align*}
\norm{ P }_{L^p(\Omega,\R^{3\times3})}+&\norm{ \Curl P }_{L^p(\Omega,\R^{3\times3})}\leq c\,\left(\norm{P}_{L^p(\Omega,\R^{3\times3})} + \norm{ \dev \Curl P }_{L^p(\Omega,\R^{3\times3})}\right)
\end{align*}
for tensor fields $P\in  W^{1,\,p}(\Curl; \Omega,\R^{3\times3})$.} These estimates also hold true for tensor fields with vanishing  tangential trace only on a relatively open (non-empty) subset $\Gamma \subseteq \partial\Omega$ of the boundary.
\end{abstract}


\msc{Primary: 35A23; Secondary: 35B45, 35Q74, 46E35.}\medskip

\keywords{$W^{1,\,p}(\Curl)$-Korn's inequality, Poincar\'{e}'s inequality, Lions lemma, Ne\v{c}as estimate, incompatibility, $\Curl$-spaces, Maxwell problems, gradient plasticity, dislocation density, relaxed micromorphic model, Cosserat elasticity, Kr\"{o}ner's incompatibility tensor, Saint-Venant compatibility, trace-free Korn's inequality, conformal mappings, conformal Killing vector field, Nye's formula.}

\section{Introduction}
Korn-type inequalities are crucial for a priori estimates in linear elasticity and fluid mechanics. They allow to bound the $L^p$-norm of the gradient $\D u$ in terms of the symmetric gradient, i.e. Korn's first inequality states
\begin{equation}\label{eq:Korn1}
 \exists\, c > 0 \ \forall\, u\in W^{1,\,p}_{0}(\Omega,\R^n): \qquad \norm{\D u}_{L^p(\Omega,\R^{n\times n})}\leq c\, \norm{\sym \D u}_{L^{p}(\Omega,\R^{n\times n})}.
\end{equation}
Generalizations to many different settings have been obtained in the literature, including the geometrically nonlinear counterpart
\cite{friesecke2002rigidity,LM2016optimalconstants,Faraco2005nonlinear}, mixed growth conditions \cite{CDM2014mixedgrowth}, incompatible  fields (also with dislocations) \cite{MSZ2014incompatible,agn_neff2015poincare,agn_neff2012canonical,agn_neff2011canonical,agn_neff2012maxwell,agn_bauer2013dev,agn_lewintan2019KornLp,agn_lewintan2019KornLpN,agn_lewintan2020KornLpN_tracefree,agn_lewintan2020generalKorn}, as well as the case of non-constant coefficients \cite{agn_neff2002korn,agn_lankeit2013uniqueness,agn_neff2014counterexamples,Pompe2003Korn} and on Riemannian manifolds \cite{ChenJost2002KornRiemann}. In this paper we focus on their improvement towards the trace-free case:
\begin{equation}\label{eq:trace-freeKorn1}
 \exists\, c > 0 \ \forall\, u\in W^{1,\,p}_{0}(\Omega,\R^n): \qquad \norm{\D u}_{L^p(\Omega,\R^{n\times n})}\leq c\, \norm{\dev_n\sym \D u}_{L^{p}(\Omega,\R^{n\times n})},
\end{equation}
where $\dev_n X\coloneqq X -\frac1n\tr(X)\cdot \id$ denotes the deviatoric (trace-free) part of the square matrix $X$. Note in passing that \eqref{eq:trace-freeKorn1} implies \eqref{eq:Korn1}.

There exist many different proofs and generalizations of the trace-free classical Korn's inequality in the literature, see \cite[Theorem 2]{Reshetnyak1970} but also \cite{Reshetnyak1994,Dain2006tracefree,FuchsSchirra2009tracefree2D,Schirra2012tracefreenD,agn_jeong2008existence,agn_bauer2013dev} as well as \cite{Wang2008tracefreeinpseudo} for trace-free Korn's inequalities in pseudo-Euclidean space and \cite{Dain2006tracefree, holst2007rough} for trace-free Korn inequalities on manifolds, \cite{BCD2017Orlicz,Fuchs2010tracefreeKorn} for trace-free Korn inequalities in Orlicz spaces and \cite{LopezGarcia2018tracefreeKorn,DingBo2020tracefreeKorn} for weighted trace-free Korn inequalities in H\"older and John domains. Such coercive inequalities found application in micro-polar Cosserat-type models \cite{agn_jeong2009numerical,Neff_Jeong_IJSS09,agn_jeong2008existence,FuchsSchirra2009tracefree2D} and general relativity \cite{Dain2006tracefree}. On the other hand, corresponding trace-free coercive inequalities for incompatible tensor fields are useful in infinitesimal gradient plasticity as well as in linear relaxed micromorphic elasticity, see \cite{agn_neff2015relaxed,agn_ghiba2017variant} but also \cite[sec. 7]{agn_bauer2013dev} and the references contained therein.

Notably, in case $n=2$, the condition $\dev_2\sym\D u\equiv0$ becomes the system of Cauchy-Riemann equations, so that the corresponding kernel is infinite-dimensional and an adequate quantitative version of the trace-free classical Korn's inequality does not hold true. Nevertheless, in \cite{FuchsSchirra2009tracefree2D} it is proved that
\begin{equation}
 \norm{\D u}_{L^p(\Omega,\R^{2\times2})}\le c\,\norm{\dev_2\sym\D u}_{L^p(\Omega,\R^{2\times2})}
\end{equation}
holds for each $u\in W^{1,\,p}_0(\Omega,\R^2)$,\footnote{A simple proof using partial integration is given in the Appendix for the case $p=2$ and all dimensions.} but, again, this result ceases to be valid if the Dirichlet conditions are prescribed only on a part of the boundary, cf.~the counterexample in \cite[sec. 6.6]{agn_bauer2013dev}.

Korn-type inequalities fail for the limiting cases $p=1$ and $p=\infty$. Indeed, from the counterexamples traced back in \cite{CFM2005counterex,dlM1964counterex,Ornstein1962,Mityagin1958} it follows that  $\int_\Omega\abs{\sym\D u}\intd{x}$ does not dominate each quantity $\int_\Omega \abs{\partial_i u_j}\intd{x}$ for any vector field $u\in W^{1,\,1}_0(\Omega,\R^n)$. Hence, also trace-free versions fail for $p=1$ and $p=\infty$. On the other hand, Poincar\'{e}-type inequalities estimating certain integral norms of the deformation $u$ in terms of the total variation of the symmetric strain tensor $\sym\D u$ are still valid. In particular, for Poincar\'{e}-type inequalities for functions of bounded deformation involving the deviatoric part of the symmetric gradient we refer to \cite{FuchsRepin2010Poincaretracefree}.\medskip

The classical Korn's inequalities need compatibility, i.e. a gradient $\D u$; giving up the compatibility necessitates controlling the distance of $P$ to a gradient by adding the incompatibility measure (the dislocation density tensor) $\Curl P$.  We showed in \cite{agn_lewintan2019KornLp} the following quantitative version of Korn's inequality for incompatible tensor fields $ P\in  W^{1,\,p}(\Curl; \Omega,\R^{3\times3})$:
\begin{align}
   \inf_{\widetilde{A}\in\so(3)}\norm{P-\widetilde{A}}_{L^p(\Omega,\R^{3\times3})}&\leq c\,\left(\norm{ \sym P }_{L^p(\Omega,\R^{3\times3})}+ \norm{ \Curl P }_{L^p(\Omega,\R^{3\times3})}\right).\label{eq:Korn_Lp_w} \\
\intertext{Note that the constant skew-symmetric matrix fields (restricted to $\Omega$) represent the elements from the kernel of the right-hand side of \eqref{eq:Korn_Lp_w}.  For compatible $P=\D u$ recover from \eqref{eq:Korn_Lp_w} the quantitative version of the classical Korn's inequality, namely for $u\in W^{1,\,p}(\Omega,\R^3)$:}
       \inf_{\widetilde{A}\in\so(3)}\norm{\D u- \widetilde{A}}_{L^p(\Omega,\R^{3\times 3})}&\leq c\, \norm{\sym \D u}_{L^{p}(\Omega,\R^{3\times 3})}\label{eq:KornQuant}\\
\intertext{and for skew-symmetric matrix fields $P=A\in\so(3)$ the corresponding Poincar\'{e} inequality for squared skew-symmetric matrix fields $A\in W^{1,\,p}(\Omega,\so(3))$ (and thus for vectors in $\R^3$):}
    \inf_{\widetilde{A}\in\so(3)}\norm{A-\widetilde{A} }_{L^p(\Omega,\R^{3\times 3})}&\leq c\, \norm{\Curl A}_{L^{p}(\Omega,\R^{3\times 3})} \le \tilde{c}\, \norm{\D A}_{L^{p}(\Omega,\R^{ {3\times3^2}})}, \label{eq:PoincareQuant}
\end{align}
where in the last step we have used that $\Curl$ consists of linear combinations from $\D$. Interestingly, for skew-symmetric $A$ also the converse is true, more precisely, the entries of $\D A$ are linear combinations of the entries from $\Curl A$, cf. e.g. \cite[Cor. 2.3]{agn_lewintan2019KornLp}:
\begin{equation}\label{eq:lincombi}
 \D A = L(\Curl A) \qquad \text{for skew-symmetric } A,
\end{equation}
 {where $L(.)$ denotes a corresponding linear operator with constant coefficients, not necessarily the same in any two places in the present paper.}
In fact, the mentioned results also hold in higher dimensions $n>3$, see \cite{agn_lewintan2019KornLpN} and the discussion contained therein. In our proof of \eqref{eq:Korn_Lp_w} we were highly inspired by a proof of \eqref{eq:KornQuant} advocated by P. G. Ciarlet and his collaborators \cite{Ciarlet2010, Ciarlet2013FAbook, Ciarlet2005korn, CMM2018,Geymonat86, DuvautLions72,Ciarlet2005intro}, which uses the Lions lemma resp.\ Ne\v{c}as estimate, the compact embedding $W^{1,\,p}\subset\!\subset L^p$ and the  representation of the second distributional derivatives of  the displacement $u$ by a linear combination of the first derivatives of the symmetrized gradient $\D u$:
\begin{equation}\label{eq:lincombi_classical}
 \D^2 u = L(\D\, \sym \D u).
 \end{equation}
It is worth mentioning that the role of the latter ingredient \eqref{eq:lincombi_classical} was taken over by \eqref{eq:lincombi} in our proof of \eqref{eq:Korn_Lp_w} in \cite{agn_lewintan2019KornLp} resp.\ \cite{agn_lewintan2019KornLpN}. In $n=3$ dimensions the relation \eqref{eq:lincombi} is an easy consequence of the so called \textit{Nye's formula}  \cite[eq.\!\! (7)]{Nye53}:
\begin{subequations}\label{eq:Nye}
\begin{align}
 \Curl A &= \tr(\D \axl A)\cdot\id- (\D \axl A)^T,\label{eq:Nye_a}\\
\shortintertext{resp.} 
\D \axl A &= \frac12 \,\tr(\Curl A)\cdot\id - (\Curl A)^T,
\end{align}
\end{subequations}
where we identify the vectorspace of skew-symmetric matrices  $\so(3)$ and $\R^3$ via ~ $\axl:\so(3)\to\R^3$ which is defined by the cross product:
\begin{equation}\label{eq:axlDef}
 A\, b \eqqcolon \axl(A)\times b \qquad \forall\, b\in\R^3,
\end{equation}
and associates with a skew-symmetric  matrix $A\in\so(3)$ the vector ~ $\axl A\coloneqq (-A_{23},A_{13},-A_{12})^T$. The relation \eqref{eq:Nye_a} admits moreover a counterpart on the group of orthogonal matrices $\operatorname{O}(3)$ and even in higher spatial dimensions, see \cite{agn_munch2008curl}. In fact, Nye's formula is (formally) a consequence of the following algebraic identity:
\begin{equation}\label{eq:prod_id-intro}
(\Anti a)\times b = b \otimes a -\skalarProd{b}{a} {\cdot} \id \qquad \forall\,  a,b\in\R^3,
\end{equation}
where the vector product of a matrix and a vector is to be seen row-wise and  ~ $\Anti: \R^3\to\so(3)$ is the inverse of ~ $\axl$. Despite the absence of the simple algebraic relations in the higher dimensional case a corresponding relation to \eqref{eq:lincombi} also holds true in $n>3$, see e.g. \cite{agn_lewintan2019KornLpN}.

Moreover, the kernel in quantitative versions of Korn's inequalities is killed by corresponding boundary conditions, namely by a vanishing trace condition $u_{|_{\partial\Omega}}=0$ in the case of \eqref{eq:KornQuant} and \eqref{eq:PoincareQuant} and by a vanishing tangential trace condition $P\times \nu~_{|_{\partial\Omega}}=0$ in the general case \eqref{eq:Korn_Lp_w}, cf. \cite{agn_lewintan2019KornLp,agn_lewintan2019KornLpN}.\medskip

The objective of the present paper is to improve on inequality \eqref{eq:Korn_Lp_w} by showing that it already suffices to consider the deviatoric (trace-free) parts on the right-hand side, hence, further contributing to the problems proposed in \cite{agn_neff2015poincare}. More precisely, the main results are
 {
\begin{theorem}
 Let $\Omega\subset\R^3$ be a bounded Lipschitz domain and $1<p<\infty$. There exists a constant $c=c(p,\Omega)>0$ such that for all $P\in L^p(\Omega,\R^{3\times3})$ we have
 \begin{align}
   \inf_{T\in K_{dS,dC}}&\norm{P-T}_{L^p(\Omega,\R^{3\times3})}\leq c\,\left(\norm{\dev \sym P }_{L^p(\Omega,\R^{3\times3})}+ \norm{\dev\Curl P }_{W^{-1,\,p}(\Omega,\R^{3\times3})}\right)\,,\label{eq:Korn_Lp_dev_quant_all}
 \end{align}
where  ~ $\dev X \coloneqq X -\frac13 \tr(X) {\cdot}\id$ ~ denotes the deviatoric part of a square tensor $X\in\R^{3\times 3}$ ~ and ~ $K_{dS,dC}$ ~ represent the kernel of the right-hand side and is given by
\begin{align}
     K_{dS,dC} &= \{T:\Omega\to\R^{3\times3}\mid  T(x)=\Anti\big(\widetilde{A}\,x+\beta\, x+b \big)+\big(\skalarProd{\axl\widetilde{A}}{x}+\gamma \big) {\cdot}\id, \notag \\ &\hspace{10em} \widetilde{A}\in\so(3), b\in\R^3, \beta,\gamma\in\R\}.\label{eq:kernel_dSdC}
 \end{align}
\end{theorem}
By killing the kernel with tangential trace conditions (note that $\dev(P\times \nu)=0$ iff $P\times\nu=0$) we arrive at the following Korn's first type inequality 
\begin{theorem}
 Let $\Omega \subset \R^3$ be a bounded Lipschitz domain and $1<p<\infty$. There exists a constant $c=c(p,\Omega)>0$ such that we have
 \begin{equation}
     \norm{ P }_{L^p(\Omega,\R^{3\times3})}\leq c\,\left(\norm{\dev \sym P }_{L^p(\Omega,\R^{3\times3})}+ \norm{ \dev\Curl P }_{L^p(\Omega,\R^{3\times3})}\right)\label{eq:Korn_Lp_thm_dSdC-intro}
 \end{equation}
 for all 
 \begin{align*}
  P\in\,&  W^{1,\,p}_0(\Curl; \Omega,\R^{3\times3})
  \coloneqq\{P\in L^p(\Omega,\R^{3\times3})\mid \Curl P\in L^p(\Omega,\R^{3\times3}), \ P\times\nu\equiv0 \text{ on }\partial\Omega\}.
 \end{align*}
\end{theorem}
The appearance of the term $\dev\Curl P$ on the right hand side of \eqref{eq:Korn_Lp_thm_dSdC-intro} would suggest to consider $p$-integrable tensor fields $P$ with `only' $p-$integrable $\dev\Curl P$. However, this would not lead to a new Banach space, since we show that for all $m\in\Z$ it holds that
\begin{equation}
 \Curl P \in W^{m,\,p}(\Omega,\R^{3\times3}) \quad \Leftrightarrow\quad \dev\Curl P \in W^{m,\,p}(\Omega,\R^{3\times3}).
\end{equation}

The estimate \eqref{eq:Korn_Lp_thm_dSdC-intro} generalizes the corresponding result in \cite{agn_bauer2013dev} from the $L^2$-setting to the $L^p$-setting, whereas the trace-free second type inequality \eqref{eq:Korn_Lp_dev_quant_all} is completely new. Generalizations to different right hand sides and higher dimensions have been obtained in the recent papers \cite{agn_lewintan2020generalKorn,agn_lewintan2020KornLpN_tracefree}. Note however that the estimates \eqref{eq:Korn_Lp_dev_quant_all} and \eqref{eq:Korn_Lp_thm_dSdC-intro}  are restricted to the case of three dimensions since the deviatoric operator acts on square matrices and only in the three-dimensional setting the matrix Curl returns a square matrix.
}

Again, for compatible $P=\D u$ we get back a tangential trace-free classical Korn inequality for the displacement gradient, namely
\begin{equation}
 \norm{\D u}_{L^p(\Omega,\R^{3\times 3})} \le c\,\norm{\dev\sym\D u}_{L^p(\Omega,\R^{3\times 3})} \quad \text{with $\D u\times \nu= 0$ on $\partial\Omega$}
\end{equation}
as well as
\begin{equation}
 \inf_{T\in K_{dS,C}}\norm{\D u-T}_{L^p(\Omega,\R^{3\times3})}\leq c\,\norm{\dev \sym \D u }_{L^p(\Omega,\R^{3\times3})}
\end{equation}
respectively
\begin{equation}\label{eq:quanti_tracefree_class}
 \norm{u-\Pi u}_{W^{1,\,p}(\Omega,\R^3)}\leq c\,\norm{\dev \sym \D u }_{L^p(\Omega,\R^{3\times3})},
\end{equation}
where $\Pi$ denotes an arbitrary projection operator from $W^{1,\,p}(\Omega,\R^3)$ onto the space of \textit{conformal Killing vectors}, here the finite dimensional kernel of $\dev\sym\D$, which is given by quadratic polynomials of the form 
\begin{align*}
 \varphi_c(x)=\skalarProd{a}{x}\,x-\frac12a\,\norm{x}^2+\Anti(b)\,x +\beta\,x +c,  \quad
  \text{with $a\coloneqq\axl\widetilde A,b,c\in\R^3$ and $\beta\in\R,$}
\end{align*}
namely the \textit{infinitesimal conformal mappings}, cf. \cite{Reshetnyak1970,Neff_Jeong_IJSS09,agn_jeong2008existence,Dain2006tracefree,Reshetnyak1994, Schirra2012tracefreenD}, see Figure 1 for an illustration in 2D.

\begin{figure}[ht!]\label{fig1}
\centering
\includegraphics[width=90mm]{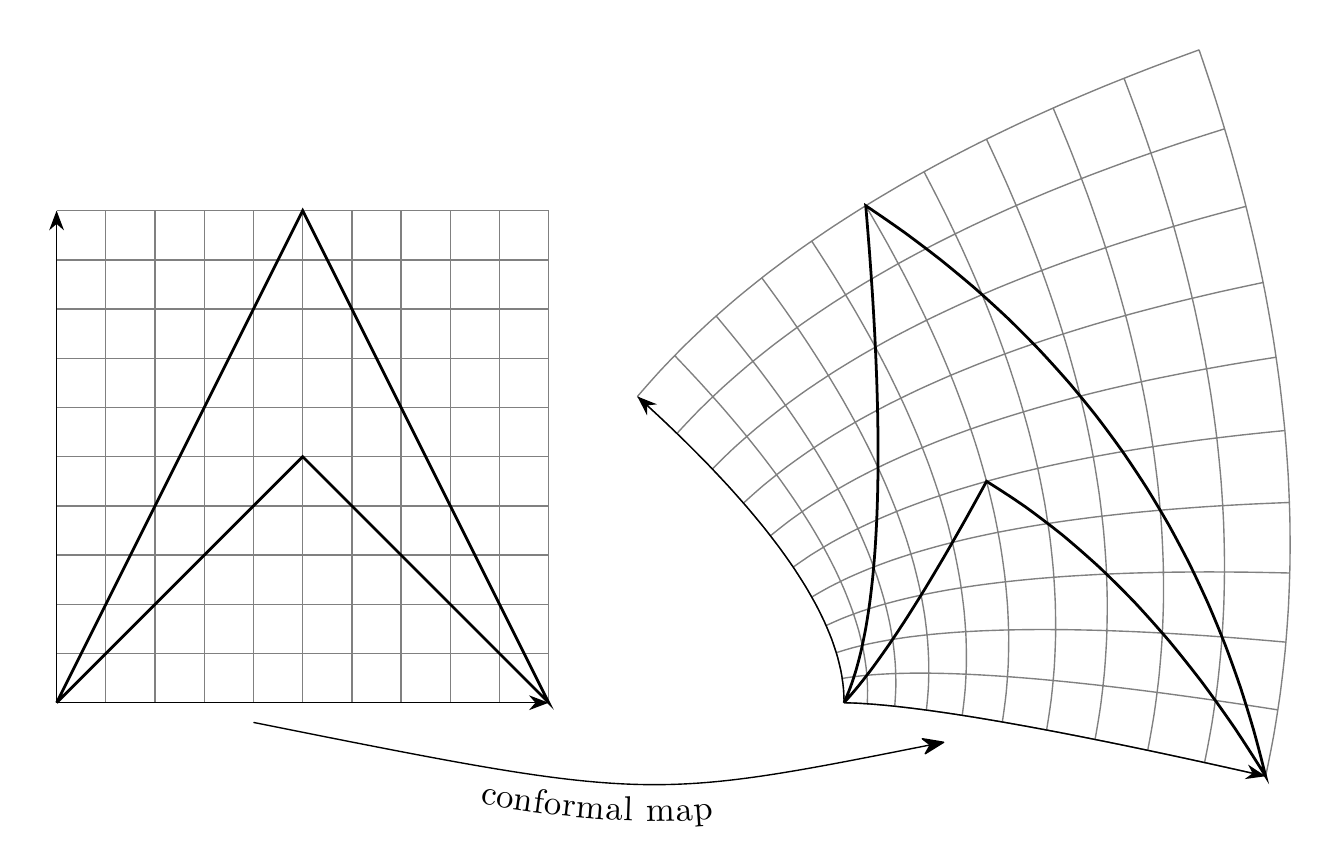}
\caption{In the planar case, the condition $\dev_2\sym \D u=0$ coincides with the Cauchy-Riemann equations for the function $u$ (see Appendix). Therefore, infinitesimal conformal mappings in 2D are holomorphic functions which preserve angles exactly. This ceases to be the case for 3D infinitesimal conformal mappings defined by $\dev_3\sym \D u=0$.}
\end{figure}

A first proof of \eqref{eq:quanti_tracefree_class}, even in all dimensions $n\geq3$, was given by Reshetnyak \cite{Reshetnyak1970} over domains which are star-like with respect to a ball. Over bounded Lipschitz domains the trace-free Korn's second inequality in all dimensions $n\geq3$, namely
\begin{align}\label{eq:tracefree-Korn2}
  \exists\, c > 0& \ \forall\, u\in W^{1,\,p}(\Omega,\R^n):\quad\norm{u}_{W^{1,\,p}(\Omega,\R^n)}\leq c\, \left(\norm{u}_{L^p(\Omega,\R^n)} + \norm{\dev_n\sym\D u}_{L^{p}(\Omega,\R^{n\times n})} \right)\,,
\end{align}
was justified by Dain \cite{Dain2006tracefree} in the case $p=2$ and by Schirra \cite{Schirra2012tracefreenD} for all $p>1$. Their proofs use again the Lions lemma and the ``higher order'' analogues of the differential relation \eqref{eq:lincombi_classical}:
\begin{equation}
 \D\Delta u = L(\D^2 \dev_n\sym \D u).
\end{equation}
However, the differential operators ~ $\sym \D$ ~ and ~ $\dev_n\sym \D$ ~ are particular cases of the so-called \textit{coercive elliptic operators} whose study began with Aronszajn \cite{Aronszajn1955}.

\medskip
Let us go back to
\begin{equation}
  \norm{ P }_{L^2(\Omega,\R^{3\times3})}\leq c\,\left(\norm{\dev \sym P }_{L^2(\Omega,\R^{3\times3})} + \norm{ \dev \Curl P }_{L^2(\Omega,\R^{3\times3})}\right)
\end{equation}
whose first proof for $P\in  W^{1,\,2}_0(\Curl; \Omega,\R^{3\times3})$  was given in \cite{agn_bauer2013dev} via the trace-free classical Korn's inequality, a Maxwell estimate and a Helmholtz decomposition and is not directly amenable to the $L^p$-case. Here, we catch up with the latter.

 {In the following section we start by summarizing the notations and collect some preliminary results from algebraic calculations which are needed in the subsequent vector calculus to establish relations of the type:
\begin{align}
 \D^3(A+\zeta\cdot\id) &= L (\D^2\dev \Curl (A+\zeta\cdot \id))
 \label{eq:relations}
\end{align}
for skew-symmetric tensor fields $A$ and scalar functions $\zeta$, where $L$ denotes a corresponding constant coefficients linear operator. Based on this ``higher order'' analogue of the differential relation \eqref{eq:lincombi} we prove our main results in the last section using a similar argumentation as in \cite{Dain2006tracefree, Schirra2012tracefreenD} which argue by the Lions lemma resp. Ne\v{c}as estimate and the compact embedding $ W^{1,p}(\Omega)\subset\!\subset L^p(\Omega)$ .
}

\section{Notations and preliminaries}
Let $n\geq2$. We consider for vectors $a,b\in\R^n$ the scalar product  $\skalarProd{a}{b}\coloneqq\sum_{i=1}^n a_i\,b_i \in \R$, the (squared) norm  $\norm{a}^2\coloneqq\skalarProd{a}{a}$ and  the dyadic product  $a\otimes b \coloneqq \left(a_i\,b_j\right)_{i,j=1,\ldots,n}\in \R^{n\times n}$. Similarly, we define the scalar product for matrices $P,Q\in\R^{n\times n}$ by $\skalarProd{P}{Q} \coloneqq\sum_{i,j=1}^n P_{ij}\,Q_{ij} \in \R$ and the (squared) Frobenius-norm by $\norm{P}^2\coloneqq\skalarProd{P}{P}$.  {We highlight by $.\cdot.$ the scalar multiplication of a scalar with a matrix, whereas matrix multiplication is denoted only by juxtaposition.}

Moreover, $P^T\coloneqq (P_{ji})_{i,j=1,\ldots,n}$ denotes the transposition of the matrix $P=(P_{ij})_{i,j=1,\ldots,n}$. The latter decomposes orthogonally into the symmetric part $\sym P \coloneqq \frac12\left(P+P^T\right)$ and the skew-symmetric part $\skew P \coloneqq \frac12\left(P-P^T\right)$. We will denote by $\so(n)\coloneqq \{A\in\R^{n\times n}\mid A^T = -A\}$ the Lie-Algebra of skew-symmetric matrices.

For the identity matrix we will write $\id$, so that the trace of a squared matrix $P$ is given by \ $\tr P \coloneqq \skalarProd{P}{\id}$. The deviatoric (trace-free) part of $P$ is given by $\dev_n P\coloneqq P -\frac1n\tr(P) {\cdot}\id$ and in three dimensions its index will be suppressed, i.e. we write $\dev$ instead of $\dev_3$.

 We will denote by $\mathscr{D}'(\Omega)$ the space of distributions on a bounded Lipschitz domain $\Omega\subset\R^n$ and by $W^{-k,\,p}(\Omega)$ the dual space of $W^{k,\,p'}_0(\Omega)$, where $p'=\frac{p}{p-1}$ is the H\"older dual exponent to $p$.

 {Throughout the paper we use $c$ as a generic positive constant, which is not necessarily the same in any two places, and we use $L(.)$ as a generic linear operator with constant coefficients, which also may differ in any two places within the paper.}

In $3$-dimensions we make use of the vector product $\times:\R^3\times\R^3 \to \R^3$.
Since the vector product $a\times .$ with a fixed vector $a\in\R^3$ is linear in the second component, there exists a unique matrix $\Anti(a)$ such that
\begin{equation}
 a\times b \eqqcolon \Anti(a) b \quad \forall\, b\in\R^3,
\end{equation}
and direct calculations show that for $a=(a_1,a_2,a_3)^T$ the matrix $\Anti(a)$ has the form
\begin{equation}\label{eq:antiform}
 \Anti(a)=\begin{pmatrix} 0 & -a_3 & a_2\\ a_3 & 0 & -a_1\\ -a_2 & a_1 & 0\end{pmatrix}.
\end{equation}
 {The inverse of $\Anti:\R^3\to\so(3)$ is denoted by $\axl:\so(3)\to\R^3$ and fulfills ~ $\axl(A)\times b = A\,b$ ~ for all skew-symmetric $(3\times3)$-matrices $A$ and vectors $b\in\R^3$. The matrix representation of the cross product }allows for a generalization towards a cross product of a matrix $P\in\R^{3\times 3}$ and a vector $b\in\R^3$ via
\begin{equation}
 P\times b\coloneqq P\,\Anti(b)\,,
\end{equation}
so, especially, for $P=\id$ it holds
\begin{equation}\label{eq:idcrossvec}
 \id\times b =\id\Anti(b)=\Anti(b) \qquad \forall\ b\in\R^3.
\end{equation}
 {We repeat the following crucial algebraic identity:
\begin{equation}\label{eq:prod_id}
(\Anti a)\times b = b \otimes a -\skalarProd{b}{a} {\cdot} \id \qquad \forall\,  a,b\in\R^3.
\end{equation}}
\begin{observation}\label{obs:traces_equal}
 For $P\in\R^{3\times 3}$ and $b\in\R^3$ we have
 \begin{equation}
   \dev(P\times b)= 0 \ \Leftrightarrow \ P\times b = 0.
 \end{equation}
\end{observation}
\begin{proof}
 We decompose $P$ into its symmetric and skew-symmetric part, i.e.,
 \[
P =  S + A = S + \Anti(a), \qquad \text{for some $S\in\Sym(3)$, $A\in\so(3)$ and with $a=\axl(A)$.}
 \]
For a symmetric matrix $S$ it holds \ $\tr(S\times b)= 0$ for any $b\in\R^3$, since\footnote{Cf.~the Appendix for component-wise calculations.}
{
\begin{align}\label{eq:trStimesb}
\tr(S\times b)&=\skalarProd{S\times b}{\id}_{\R^{3\times 3}}=\skalarProd{S\Anti(b)}{\id}_{\R^{3\times 3}} = -\skalarProd{S}{\Anti(b)}_{\R^{3\times 3}}\overset{S\in\Sym(3)}{=}0.
\end{align}
}
Thus, using the decomposition $P = S + \Anti(a)$, we have:
\begin{align}
 \dev(P\times b) &= P\times b -\frac13\tr(P\times b) {\cdot}\id \overset{\eqref{eq:trStimesb}}{=} P\times b -\frac13\tr((\Anti a)\times b) {\cdot}\id\notag\\
 &\overset{\mathclap{\eqref{eq:prod_id}}}{=}\ P\times b - \frac13\tr(b\otimes a -\skalarProd{b}{a} {\cdot}\id) {\cdot}\id = P\times b +\frac23\skalarProd{a}{b} {\cdot}\id.\label{eq:firststepdevPtimesb}
\end{align}
Moreover, for any matrix $P\in\R^{3\times3}$ we note that
\begin{equation}\label{eq:crossprod_scalarprod}
 (P\times b)\,b = (P\Anti(b)) b = P(\Anti(b)\, b)=P(b\times b) =0.
\end{equation}
Thus, we obtain
\begin{equation}\label{eq:hilfs_schritt2}
\skalarProd{b}{\dev(P\times b)\,b} \overset{\eqref{eq:firststepdevPtimesb}}{=}\skalarProd{b}{\big(P\times b +\frac23\skalarProd{a}{b} {\cdot}\id\big)\, b} \overset{\eqref{eq:crossprod_scalarprod}}{=} \frac23\skalarProd{a}{b}\,\norm{b}^2,
\end{equation}
and the conclusion follows from the identity
\begin{align}\label{eq:PtbedevPtb}
 \norm{b}^2 {\cdot} P\times b &\overset{\eqref{eq:firststepdevPtimesb}}{=} \norm{b}^2  {\cdot} \dev(P\times b)-\frac23\norm{b}^2\skalarProd{a}{b} {\cdot}\id
 \overset{\eqref{eq:hilfs_schritt2}}{=} \norm{b}^2  {\cdot} \dev(P\times b) - \skalarProd{b}{\dev(P\times b)\,b} {\cdot}\id.
 \end{align}
 An application of the Cauchy-Bunyakovsky-Schwarz inequality on the right hand side of \eqref{eq:PtbedevPtb} shows that
 \begin{equation}
 \norm{\dev(P\times b)} \le \norm{P\times b} \le \left(1+\sqrt{3}\right)\norm{\dev(P\times b)}\,.
 \end{equation}
\end{proof}

\begin{observation}\label{obs:zwei}
 Let $a\in\R^3$ and $\alpha\in\R$, then
 \[
 (\Anti(a)+\alpha\cdot\id)\times b= 0 \text{ ~ for $b\in\R^3\backslash\{0\}$}  \quad \Rightarrow \quad a=0 \text{ and } \alpha=0.
 \]

\end{observation}
\begin{proof}
 By \eqref{eq:prod_id} and \eqref{eq:idcrossvec} we have:
 \begin{equation}\label{eq:firststep}
  0 = (\Anti(a)+\alpha\cdot\id)\times b = b\otimes a -\skalarProd{b}{a} {\cdot}\id + \alpha\cdot \Anti(b).
 \end{equation}
Taking the trace on both sides we obtain
\[
0 = \tr( b\otimes a -\skalarProd{b}{a} {\cdot}\id + \alpha\cdot \Anti(b))  = \skalarProd{a}{b}-3\,\skalarProd{a}{b} = -2\,\skalarProd{a}{b}.
\]
Thus, reinserting $\skalarProd{b}{a} =0$ in \eqref{eq:firststep} and applying $\sym$ on both sides, this implies $\sym(b\otimes a)=0$. {Since
\begin{equation}
 \norm{\sym(a\otimes b)}^2 = \frac12\norm{a}^2\norm{b}^2+\frac12\skalarProd{a}{b}^2
\end{equation}
and $b\neq0$
}
we must have $a=0$. Hence, by \eqref{eq:firststep} also $\alpha=0$.
\end{proof}

Formally the gradient and the curl of a vector field $a:\Omega\to\R^3$ can be seen as
$$
\D a = a\otimes \nabla \quad \text{and}\quad \curl a = a\times (-\nabla).
$$
The latter also generalizes to $(3\times3)$-matrix fields $P:\Omega\to\R^{3\times3}$ row-wise:\footnote{In the literature, the matrix $\Curl$ operator is sometimes defined as our transposed $(\Curl P)^T$, cf. Ciarlet \cite[Problem 6.18-4]{Ciarlet2013FAbook}.}
\begin{equation}\label{eq:Curl_def}
 \Curl P = P \times (-\nabla) = \begin{pmatrix} (P^T e_1)^T \\ (P^Te_2)^T \\ (P^Te_3)^T \end{pmatrix} \times (-\nabla)  =\begin{pmatrix} (\operatorname{curl}\,(P^Te_1))^T \\(\operatorname{curl}\,(P^Te_2))^T \\(\operatorname{curl}\,(P^Te_3))^T  \end{pmatrix}\in\R^{3\times3}.
\end{equation}
 {Replacing $b$ by $\nabla$ in \eqref{eq:prod_id} we obtain Nye's formulas
\begin{subequations}\label{eq:Nye-prelim}
 \begin{align}
  \Curl A &= \tr(\D \axl A)\cdot\id- (\D \axl A)^T,\label{eq:Nye-prelim_a}\\
\shortintertext{and}
\D \axl A &= \frac12 \,\tr(\Curl A)\cdot\id - (\Curl A)^T\label{eq:Nye-prelim_b}
 \end{align}
\end{subequations}
for all skew-symmetric $(3\times3)$-matrix fields $A$.
}
\begin{remark}
 Formal calculations (e.g. replacing $b$ by $\nabla$) have to be performed very carefully. Indeed, they are allowed in algebraic identities but fail, in general, for implications, e.g. for $A\in\so(3)$ and $b\in\R^3$  we have ~ $A\times b = 0$ ~ if and only if ~ $\dev(A\times b)= 0$, since the following expression holds true, cf. Observation \ref{obs:traces_equal} and \eqref{eq:PtbedevPtb}:
 \begin{equation}\label{eq:AtbedevAtb}
 \norm{b}^2 {\cdot} A\times b =\norm{b}^2  {\cdot} \dev(A\times b) - \skalarProd{b}{\dev(A\times b)\,b} {\cdot}\id\,.
 \end{equation}
 However,~ $\dev(\Curl A)=\dev(A\times(-\nabla))=0$ ~ does not imply already that ~ $\Curl A = A\times(-\nabla)= 0$, due to the counterexample ~ $A=\Anti(x)$, ~ since  by Nye's formula \eqref{eq:Nye-prelim} we have~ $\Curl(\Anti(x))=2\cdot\id$. ~ Of course, we can interpret \eqref{eq:AtbedevAtb} also in the sense of vector calculus, which gives then an expression for ~ $\Delta\Curl A$ ~ in terms of the second distributional derivatives of ~ $\dev(\Curl A)$, ~ but, the latter would have no meaning for the relation of ~ $\Curl A$ ~ and ~ $\dev\Curl A$.
\end{remark}

\begin{lemma} \label{lem:lin_combi} Let $A\in\mathscr{D}'(\Omega,\so(3))$ and $\zeta\in\mathscr{D}'(\Omega,\R)$. Then
\begin{thmenum}
 \item the entries of  $\D^2 (A+\zeta\cdot\id)$ are  linear combinations of the entries of $\D\Curl (A+\zeta\cdot\id)$.\label{lin_combi_trace_a}
 \item the entries of  $\D^2 A$ are  linear combinations of the entries of $\D\dev\Curl A$.\label{lin_combi_trace_b}
 \item  the entries of $\D^3 (A+\zeta\cdot\id)$ are linear combinations of the entries of $\D^2  \dev\Curl(A+\zeta\cdot\id)$.\label{lin_combi_trace_c}
\end{thmenum}

\end{lemma}
\begin{proof}
Observe that applying \eqref{eq:idcrossvec} to the vector field $\nabla \zeta$ we obtain:
\begin{equation}\label{eq:skalarID}
 \Curl (\zeta\cdot\id) \overset{\eqref{eq:Curl_def}}{=} \id \times (-\nabla \zeta) \overset{\eqref{eq:idcrossvec}}{=} -\Anti(\nabla \zeta).
\end{equation}
Let us first start by proving part \ref{lin_combi_trace_b}. From Nye's formula \eqref{eq:Nye-prelim_a} we obtain
\begin{equation}
 \dev \Curl A = \frac13\tr(\D\axl A) {\cdot}\id - (\D\axl A)^T\label{eq:devNye}
\end{equation}
so that taking the $\Curl$ of the transpositions on both sides gives
\begin{equation}\label{eq:daraufwirdaxlangewendet}
 \Curl([ \dev \Curl A]^T) \underset{\eqref{eq:devNye}}{\overset{\Curl \circ \D\, \equiv 0}{=} } \frac13\Curl(\tr(\D\axl A) {\cdot}\id) \overset{\eqref{eq:skalarID}}{=} - \frac13\Anti(\nabla \tr(\D\axl A))\,.
\end{equation}
In other words, we have that $\Curl([ \dev \Curl A]^T)\in\so(3)$, and applying $\axl$ on both sides of \eqref{eq:daraufwirdaxlangewendet} we obtain
\begin{equation}\label{eq:lin_combi_dev}
 \nabla \tr(\D\axl A) = -3\axl(\Curl([ \dev \Curl A]^T)) = L_0(\D\dev \Curl A).
\end{equation}
Taking the  $\partial_j$-derivative of \eqref{eq:devNye} for $j=1,2,3$ we conclude
\begin{equation}
 \partial_j(\D\axl A)^T \overset{\eqref{eq:devNye}}{=} \frac13\partial_j\tr(\D\axl A)-\partial_j\dev \Curl A  \overset{\eqref{eq:lin_combi_dev}}{=}\widetilde L_0(\D\dev \Curl A)\,,
\end{equation}
which establishes part \ref{lin_combi_trace_b}, namely $\D^2 A = L_2(\D(\dev \Curl A))$ for skew-symmetric tensor fields $A$.\medskip

The proof of part \ref{lin_combi_trace_a} is divided into the following two key observations:

\begin{enumerate}[label=(a.\roman*)~]
\begin{multicols}{2}
 \item $\D^2 \zeta = \widetilde L_1(\D\Curl (A+\zeta\cdot\id))$,\label{eq:firstID}
 \item $\D^2 A = \widetilde L_2(\D\Curl (A+\zeta\cdot\id))$.\label{eq:secondID}
\end{multicols}
\end{enumerate}
\noindent
To show that each entry of the Hessian matrix $\D^2 \zeta$ is a linear combination of the entries of $\D\Curl (A+\zeta\cdot\id)$ we make use of the second-order differential operator $\inc$ given for $B\in \mathscr{D}'(\Omega,\R^{3\times 3})$  via\footnote{See Kr\"oner \cite[\S 8]{Kroener1954} for a component-wise expression of the incompatibility operator $\inc$.} \begin{equation}\inc B \coloneqq \Curl ([\Curl B]^T)\end{equation}
so that
\begin{align}
 \inc (\zeta\cdot\id) &= \Curl ([\Curl(\zeta\cdot\id)]^T)\overset{\eqref{eq:skalarID}}{=}\Curl (-[\Anti(\nabla \zeta)]^T) = \Curl (\Anti(\nabla \zeta))\notag\\
 &\overset{\mathclap{\eqref{eq:Nye-prelim_a}}}{=}\ \ \tr(\D\nabla\zeta)\cdot\id-(\D\nabla \zeta)^T
 = \Delta \zeta\cdot \id - \D^2 \zeta \in \Sym(3)\label{eq:inc_id_scalar}
 \end{align}
 is symmetric. On the other hand, for a skew-symmetric matrix field $A\in \mathscr{D}'(\Omega,\so(3))$ we have that
\begin{align}
 \inc A &= \Curl ([\Curl A]^T) \overset{\eqref{eq:Nye-prelim}}{=} \Curl ( \tr (\D \axl A)\cdot \id - \D \axl A )\notag\\
 &\overset{\mathclap{\Curl \circ \D\, \equiv 0}}{=}\qquad \Curl ( \tr (\D \axl A)\cdot \id) \overset{\eqref{eq:skalarID}}{=}-\Anti (\nabla\tr (\D \axl A))\in\so(3)\label{eq:inc_id_skew}
\end{align}
is skew-symmetric. Hence,
\begin{equation}
 \sym (\inc (A+\zeta\cdot\id)) = \Delta \zeta\cdot \id - \D^2 \zeta \quad \text{and}\quad \tr (\inc (A+\zeta\cdot\id)) = 2\,\Delta\zeta .
\end{equation}
In other words, the entries of the Hessian matrix of $\zeta$ are linear combinations of entries from $\inc (A+\zeta\cdot\id)$:
\begin{align}\label{eq:d2zeta1}
 \D^2\zeta &=  \Delta \zeta\cdot \id - \sym (\inc (A+\zeta\cdot\id)) = \frac12\tr (\inc (A+\zeta\cdot\id)) {\cdot\id} -  \sym (\inc (A+\zeta\cdot\id))\notag\\
 &= \widetilde  L_1(\D\Curl (A+\zeta\cdot\id)),
\end{align}
where we have used that the entries of $\inc B$ are, of course, linear combinations of entries of $\D \Curl B$.\medskip

To establish \ref{eq:secondID} from \ref{eq:firstID}, recall that for a skew-symmetric matrix field $A$ the entries of $\D A$ are linear combinations of the entries from $\Curl A$:
\begin{align}\label{eq:zuDiffenrenzieren}
  \D A &\overset{\eqref{eq:lincombi}}{=} L (\Curl A) = L(\Curl(A+\zeta\cdot\id))- L(\Curl (\zeta\cdot \id)) 
  \overset{\eqref{eq:skalarID}}{=}  L(\Curl(A+\zeta\cdot\id)) + L(\Anti (\nabla \zeta)).
\end{align}
We conclude by taking the $\partial_j$-derivative of \eqref{eq:zuDiffenrenzieren} for $j=1,2,3$, namely
\begin{equation*}
 \partial_j \D A = L(\partial_j\Curl(A+\zeta\cdot\id)) + L(\partial_j\Anti (\nabla \zeta)) \overset{\text{\ref{eq:firstID}}}{=} \widetilde L_3(\D\Curl(A+\zeta\cdot\id)).
\end{equation*}

Finally, we establish part \ref{lin_combi_trace_c} arguing in a similar way by showing the following linear combinations:
\begin{enumerate}[label=(c.\roman*)~]
\begin{multicols}{2}
 \item $\D^2 \zeta = \widetilde L_4(\D\dev\Curl (A+\zeta\cdot\id))$,\label{eq:firstID_c}
 \item $\D^3 A =\widetilde L_7(\D^2\dev\Curl (A+\zeta\cdot\id))$.\label{eq:secondID_c}
\end{multicols}
\end{enumerate}
Regarding \eqref{eq:skalarID} and \eqref{eq:Nye-prelim} we have
\begin{align}\label{eq:Nye_devCurl}
 \dev\Curl(A+\zeta\cdot\id) &\overset{\eqref{eq:skalarID}}{=} \dev[\Curl A -\Anti(\nabla \zeta)] =\dev\Curl A -\Anti(\nabla \zeta)\notag\\&\overset{\eqref{eq:Nye-prelim}}{=} \frac13\tr(\D\axl A) {\cdot}\id-(\D\axl A)^T -\Anti(\nabla \zeta).
\end{align}
Transposing and taking the $\Curl$ on both sides yields
\begin{equation}\label{eq:Nye_CurldevCurl}
 \Curl ([\dev\Curl(A+\zeta\cdot\id)]^T) \underset{\Curl \circ \D\, \equiv 0}{\overset{\eqref{eq:skalarID},\, \eqref{eq:Nye-prelim}}{=}} -\frac13\underset{\in\so(3)}{\underbrace{\Anti(\nabla \tr(\D\axl A))}}+ \underset{\in\Sym(3)}{\underbrace{\Delta \zeta\cdot \id - \D^2\zeta}}\vspace{-2ex}
\end{equation}
and we obtain, similar to the decomposition in \eqref{eq:d2zeta1}:
\begin{align}\label{eq:in_combi_dd_zeta}
\D^2\zeta &= \frac12\tr(\Curl ([\dev\Curl(A+\zeta\cdot\id)]^T) ) {\cdot\id}- \sym(\Curl ([\dev\Curl(A+\zeta\cdot\id)]^T) )\notag\\
& = \widetilde L_4(\D\dev\Curl(A+\zeta\cdot\id)).
\end{align}
On the other hand, taking $\inc$ of the transpositions on both sides of \eqref{eq:Nye_devCurl} gives
\begin{equation}
 \inc([\dev\Curl(A+\zeta\cdot\id)]^T)\underset{\eqref{eq:inc_id_skew}}{\overset{\eqref{eq:inc_id_scalar}}{=}} \frac13\Delta \tr(\D\axl A)\cdot\id - \frac13 \D^2\tr(\D\axl A)-\Anti(\nabla \Delta \zeta)\,,
\end{equation}
yielding the relation
\begin{align}\label{eq:Nye_combi_dd_trA}
 \D^2 \tr(\D\axl A) &= \frac32\tr(\inc([\dev\Curl(A+\zeta\cdot\id)]^T)) {\cdot\id}\notag\\
 &\hspace*{7em}-\sym(\inc([\dev\Curl(A+\zeta\cdot\id)]^T))  \notag\\
 &=\widetilde L_5(\D^2  \dev\Curl(A+\zeta\cdot\id)).
\end{align}
Considering the second distributional derivatives in \eqref{eq:Nye_devCurl}  we conclude
\begin{equation*}
\begin{split}
\D^3  \axl A &= \frac13 \D^2 \tr(\D\axl A) {\cdot} \id -\D^2([\dev\Curl(A+\zeta\cdot\id)]^T)+\D^2\Anti(\nabla \zeta)\\
&\underset{\eqref{eq:Nye_combi_dd_trA}}{\overset{\eqref{eq:in_combi_dd_zeta}}{=}}\widetilde L_6(\D^2  \dev\Curl(A+\zeta\cdot\id)).\qedhere
\end{split}
\end{equation*}
\end{proof}
\begin{remark}\label{rem:inc}
In the above proof we have used that the second-order differential operator $\inc$ does not change the symmetry property after application on square matrix fields, cf.~the Appendix. Further properties are collected e.g. in \cite[Appendix]{agn_neff2014unifying}, \cite[Sec. 2]{Amrouche2006inc} and \cite[Sec. 6.18]{Ciarlet2013FAbook}.

The incompatibility operator $\inc$ arises in dislocation models, e.g., in the modeling of elastic materials with dislocations or in the modeling of dislocated crystals, since the strain cannot be a symmetric gradient of a vector field as soon as dislocations are present and the notion of incompatibility is at the basis of a new paradigm to describe the inelastic effects, cf. \cite{agn_ebobisse2017fourth,Amstutz2019incompatibility,Amstutz2016analysis,Maggiani2015incompatible}, cf.~the Appendix for further comments. Moreover, the equation ~ $\inc\sym e\equiv 0$ ~ is equivalent to the \textit{Saint-Venant compatibility condition}\footnote{Those compatibility conditions are contained in the third appendix \S 32 p. 597 et seq. of the third edition of the lecture notes  \textit{R\'{e}sistance des corps solides} given by Navier and extended with several notes and appendices by Barr\'{e} de Saint-Venant and published as \textit{R\'{e}sum\'{e} des Le\c{c}ons donn\'{e}es \`{a} l'\'{E}cole des Ponts et Chauss\'{e}es sur l'Application de la M\'{e}canique}, vol. I, Paris, 1864. Their coordinate-free version can be found in Lagally's monograph on vector calculus from 1928 \cite[Ziff. 191]{Lagally1928} where it reads:
\[
 \nabla \times (\sym \D u ) \times \nabla \equiv 0
\]
and formally follows from the definitions of those operators, see \cite[Ziff. 191]{Lagally1928}, since
\[
 \nabla \times (\sym \D u ) \times \nabla = \frac 12 \nabla \times (\nabla \otimes u + u \otimes \nabla ) \times \nabla = \frac 12 [(\nabla\times \nabla)\otimes u\times \nabla + \nabla\times u \otimes (\nabla \times \nabla) ] \equiv 0.
\]
\label{footnote:inc}} defining the relation between the symmetric strain $\sym e$ and the displacement vector field $u$:
\begin{equation}
\inc \sym e \equiv 0 \quad \Leftrightarrow \quad \sym e = \sym \D u
\end{equation}
over simply connected domains, cf. \cite{Amrouche2006inc,Maggiani2015incompatible}. In the appendix we show that the operators $\inc$ and $\sym$ can be interchanged, so that
\begin{equation}
 \inc \sym e = \sym \inc e = \sym \Curl ([\Curl e]^T).
 \end{equation}
Investigations over multiply connected domains can be found e.g. in \cite{Ting1977StVenant,Geymonat2005Venant}.
\end{remark}

Returning to our proof, a crucial ingredient in our following argumentation is
\begin{theorem}[Lions lemma and Ne\v{c}as estimate]  \label{th:necas_estimate}
Let $\Omega\subset\R^n$ be a bounded  Lipschitz  domain. Let $m \in \Z$ and $p \in (1, \infty)$.
Then $f \in  \mathscr{D}'(\Omega,\R^d)$ and $\D f \in W^{m-1,\,p}(\Omega,\R^{d\times n})$ imply
$f \in W^{m,\,p}(\Omega,\R^d)$.
Moreover,
\begin{equation}  \label{eq:necas_m_p}
 \norm{ f}_{W^{m,\,p}(\Omega,\R^d)} \le c\,\left(\norm{ f}_{W^{m-1,\,p}(\Omega,\R^d)} + \norm{ \D f }_{W^{m-1,\,p}(\Omega,\R^{d\times n})}\right),
\end{equation}
with a constant $c=c(m,p,n,d,\Omega)>0$.
\end{theorem}
For the proof we refer to \cite[Proposition 2.10 and Theorem 2.3]{AG90}, \cite{BS90}. However, since we are dealing with higher order derivatives we also need a ``higher order'' version of the Lions lemma resp.\ Ne\v{c}as estimate.
\begin{corollary}\label{cor:LionsNecas_k}
 Let $\Omega\subset\R^n$ be a bounded  Lipschitz  domain, $m \in \Z$ and $p \in (1, \infty)$. Denote by  $\D^k f$ the collection of all distributional derivatives of order $k$. Then $f \in  \mathscr{D}'(\Omega,\R^d)$ and $\D^k f \in W^{m-k,\,p}(\Omega,\R^{d\times n^k})$ imply
$f \in W^{m,\,p}(\Omega,\R^d)$.
Moreover, 
\begin{equation}  \label{eq:necask_m_p}
 \norm{ f}_{W^{m,\,p}(\Omega,\R^d)} \le c\,\left(\norm{ f}_{W^{m-1,\,p}(\Omega,\R^d)} + \norm{ \D^k f }_{W^{m-k,\,p}(\Omega,\R^{d\times n^k})}\right),
\end{equation}
with a constant $c=c(m,p,n,d,\Omega)>0$.
\end{corollary}
\begin{proof}The assertion $f \in W^{m,\,p}(\Omega,\R^d)$ and the estimate \eqref{eq:necask_m_p} follow by inductive application of Theorem \ref{th:necas_estimate} to $\D^l f$ with $l=k-1,k-2,\ldots,0$. Indeed, starting by applying Theorem \ref{th:necas_estimate} to $\D^{k-1} f$ gives $\D^{k-1} f \in W^{m-k+1,\,p}(\Omega,\R^{d\times n^{k-1}})$ as well as
 \begin{align}
   \norm{ \D^{k-1} f}_{W^{m-k+1,\,p}(\Omega,\R^{d\times n^{k-1}})} 
   &\le c\,\left(\norm{\D^{k-1} f}_{W^{m-k,\,p}(\Omega,\R^{d\times n^{k-1}})} + \norm{ \D^k f }_{W^{m-k,\,p}(\Omega,\R^{d\times n^k})}\right)\notag\\
   &\le c\,\left(\norm{f}_{W^{m-1,\,p}(\Omega,\R^d)} + \norm{ \D^k f }_{W^{m-k,\,p}(\Omega,\R^{d\times n^k})}\right).\label{eq:necas_hilf_k-1}
 \end{align}
 Now, we can apply Theorem \ref{th:necas_estimate} to $\D^{k-2} f$ to deduce $\D^{k-2} f \in W^{m-k+2,\,p}(\Omega,\R^{d\times n^{k-2}})$ and moreover
  \begin{align}
   \norm{ \D^{k-2} f}_{W^{m-k+2,\,p}(\Omega,\R^{d\times n^{k-2}})}
   &\le c\,\left(\norm{\D^{k-2} f}_{W^{m-k+1,\,p}(\Omega,\R^{d\times n^{k-1}})} + \norm{ \D^{k-1} f }_{W^{m-k+1,\,p}(\Omega,\R^{d\times n^{k-1}})}\right)\notag\\
   &\le c\,\left(\norm{f}_{W^{m-1,\,p}(\Omega,\R^d)} + \norm{ \D^{k-1} f }_{W^{m-k+1,\,p}(\Omega,\R^{d\times n^{k-1}})}\right)\notag\\
   &\overset{\mathclap{\eqref{eq:necas_hilf_k-1}}}{\le}c\,\left(\norm{f}_{W^{m-1,\,p}(\Omega,\R^d)} + \norm{ \D^k f }_{W^{m-k,\,p}(\Omega,\R^{d\times n^k})}\right). \label{eq:necas_hilf_k-2}
 \end{align}
Consequently, for all $l=k-1,k-2,\ldots,0$ we deduce $\D^l f \in W^{m-l,\,p}(\Omega,\R^{d\times n^l})$ as well as
  \begin{equation}
   \norm{ \D^l f}_{W^{m-l,\,p}(\Omega,\R^{d\times n^l})} \le c\,\left(\norm{f}_{W^{m-1,\,p}(\Omega,\R^d)} + \norm{ \D^k f }_{W^{m-k,\,p}(\Omega,\R^{d\times n^k})}\right). \label{eq:necas_hilf_l}
 \end{equation}
 \end{proof}

\begin{remark}
 The need to consider higher order derivatives is indicated by the appearance of linear terms in the kernel of Korn's quantitative versions, similar to the situation at the classical trace-free Korn inequalities \cite{Dain2006tracefree,Schirra2012tracefreenD}. In our case we have:
\end{remark}

{
\begin{lemma}\label{lem:Kernel}
Let $A\in L^p(\Omega,\so(3))$ and $\zeta\in L^p(\Omega,\R)$. Then we have in the distributional sense
\begin{thmenum}
 \item\label{kernel_a} ~ 
 $\Curl(A+\zeta\cdot\id)\equiv 0$  if and only if $A+\zeta\cdot\id=\Anti(\widetilde{A}\,x+b) +(\skalarProd{\axl \widetilde{A}}{x}+\beta)\cdot\id$ a.e. on $\Omega$,
  \item\label{kernel_b} ~ $\dev\Curl A \equiv 0$ if and only if $A=  \Anti(\beta\,x+b)$ a.e. on $\Omega$,
 \item\label{kernel_c} ~ $\dev\Curl(A+\zeta\cdot\id)\equiv 0$ if and only if $A+\zeta\cdot\id= \Anti\big(\widetilde{A}\,x+\beta\, x+b \big)+\big(\skalarProd{\axl\widetilde{A}}{x}+\gamma \big)\cdot\id$ a.e. on $\Omega$,
 \end{thmenum}
 with constant $\widetilde{A}\in\so(3)$, $b\in\R^3$, $\beta,\gamma\in\R$.
\end{lemma}
}

\begin{proof}
Although the deductions have already been partially indicated in the literature, cf. e.g. \cite[sec. 3.4]{agn_neff2009new} and \cite{Reshetnyak1970,Dain2006tracefree,agn_bauer2013dev,Reshetnyak1994}, we include it here for the sake of completeness. The ``if''-parts are seen by direct calculations, cf. the relations \eqref{eq:Nye-prelim} and \eqref{eq:skalarID}:
\begin{thmenum}
 \item ~ $\Curl(\Anti\big(\widetilde{A}\,x+b\big)+\big(\skalarProd{\axl\widetilde{A}}{x}+\beta \big) {\cdot}\id) = \widetilde{A}-\Anti(\axl \widetilde{A})\equiv0$,
 \item ~ $\dev\Curl(\Anti(\beta\,x+b))=\dev(\tr(\beta\cdot\id)\cdot\id-\beta\cdot\id)=\dev(2\,\beta\cdot\id)\equiv0$,
 \item ~ $\dev\Curl(\Anti\big(\widetilde{A}\,x+\beta\, x+b \big)+\big(\skalarProd{\axl\widetilde{A}}{x}+\gamma \big) {\cdot}\id)$\\\hspace*{2em} $= \dev\big(\widetilde{A} +2\,\beta {\cdot}\id-\Anti(\axl \widetilde{A}) \big)\equiv 0$.
\end{thmenum}
Now, we focus on the ``only if''-directions, starting with
 \[
  \Curl (A +  \zeta\cdot\id) \equiv 0 \quad \overset{\eqref{eq:skalarID}}{\Longleftrightarrow} \quad \Anti(\nabla\zeta) = \Curl A \overset{\eqref{eq:Nye-prelim}}{=} \tr(\D\axl A) {\cdot}\id-(\D\axl A)^T.
 \]
Taking the trace on both sides we obtain \ $\tr(\D\axl A)=0$ \ and consequently
\begin{equation}\label{eq:hilfsKern}
 \Anti(\nabla\zeta) =-(\D\axl A)^T,
\end{equation}
hence  \ $\sym(\D\axl A)=0$. By the classical Korn's inequality \eqref{eq:KornQuant} it follows that there exists a constant skew-symmetric matrix $\widetilde{A}\in\so(3)$ so that \ $\D \axl A \equiv \widetilde{A}$, \ which implies $A=\Anti(\widetilde{A}x+b)$ with $b\in\R^3$. Furthermore, by \eqref{eq:hilfsKern} we obtain
\[
\Anti(\nabla\zeta) = \widetilde{A} \quad \Rightarrow \quad \zeta = \skalarProd{\axl \widetilde{A}}{x}+\beta \quad \text{with } \beta\in\R,
\]
which establishes \ref{kernel_a}.

\medskip

For part \ref{kernel_b} we start with the relation $\dev\Curl A\equiv0$ in \eqref{eq:daraufwirdaxlangewendet} and have
\begin{equation}
 \Anti(\nabla \tr(\D\axl A))\equiv0 \quad \Rightarrow \quad \nabla\tr(\D\axl A)\equiv0,
\end{equation}
so that
\begin{equation}
 \frac13\tr(\D\axl A) =\beta
\end{equation}
for some $\beta\in\R$. Reinserting in the deviatoric counterpart of Nye's formula \eqref{eq:devNye} gives
\begin{equation}
0 = \beta\cdot\id-(\D\axl A)^T \quad\text{resp.}\quad \D\axl A = \beta\cdot\id\quad \Rightarrow\quad \axl A = \beta\, x + b
\end{equation}
for some $b\in\R^3$ and thus ~ $A = \Anti(\beta\, x + b)$.\medskip

Finally, for part \ref{kernel_c}, let now \ $\dev\Curl(A+\zeta\cdot\id)\equiv 0$. \ Then considering the skew-symmetric parts of \eqref{eq:Nye_CurldevCurl} we obtain
\[
\Anti(\nabla\tr(\D\axl A)) \equiv 0 \quad \Rightarrow \quad \nabla\tr(\D\axl A) \equiv 0.
\]
Hence, again
\begin{equation}\label{eq:tr_axl_a_2}
 \frac13\tr(\D\axl A) = \beta
\end{equation}
for some $\beta\in\R$, so that considering the symmetric parts of \eqref{eq:Nye_devCurl} we get
\begin{equation}\label{eq:sym_d_axl}
 0 = \frac13\tr(\D\axl A) {\cdot}\id -\sym(\D\axl A) \overset{\eqref{eq:tr_axl_a_2}}{=} \beta {\cdot}\id - \sym(\D\axl A).
\end{equation}
In other words, we have
\[
\sym(\D(\axl A -\beta\,x)) \equiv 0
\]
and by \eqref{eq:KornQuant},  it follows that \ $\D(\axl A -\beta\,x)$ \ must be a constant skew-symmetric matrix. Thus
\begin{equation}\label{eq:darstellung_axl}
 \axl A = \widetilde{A}\,x  + \beta\, x+ b
\end{equation}
for some $\widetilde{A}\in\so(3)$, $b\in\R^3$ and $\beta\in\R$.  Furthermore, by \eqref{eq:Nye_devCurl} we have
\[
 \Anti(\nabla \zeta) \overset{\eqref{eq:Nye_devCurl}}{=} \skew(\D\axl A) \overset{\eqref{eq:darstellung_axl}}{=} \widetilde{A}
\]
so that $\zeta$ is of the form
\begin{equation}\label{eq:darstellung_zeta}
\zeta = \skalarProd{\axl \widetilde{A}}{x} + \gamma
\end{equation}
for some $\gamma\in\R$, and we arrive at \ref{kernel_c}:
\[
A+\zeta\cdot\id \underset{\eqref{eq:darstellung_zeta}}{\overset{\eqref{eq:darstellung_axl}}{=}} \Anti\big(\widetilde{A}\,x+\beta\, x+b \big)+\big(\skalarProd{\axl\widetilde{A}}{x}+\gamma \big) {\cdot}\id.\qedhere
\]
\end{proof}

We are now prepared to proceed as in the proof of the generalized Korn inequality for incompatible tensor fields {.}

\section{Main results}
We will make use of the Banach space
\begin{subequations}
\begin{align}
  W^{1,\,p}(\Curl; \Omega,\R^{3\times3}) &\coloneqq \{P\in L^p(\Omega,\R^{3\times3})\mid \Curl P \in L^p(\Omega,\R^{3\times3})\}
 \shortintertext{equipped with the norm}
 \norm{P}_{ W^{1,\,p}(\Curl; \Omega,\R^{3\times3})}&\coloneqq \left(\norm{P}^p_{L^p(\Omega,\R^{3\times3})} + \norm{\Curl P}^p_{L^p(\Omega,\R^{3\times3})} \right)^{\frac{1}{p}},
\end{align}
\end{subequations}
as well as its subspace
\begin{align*}
  W^{1,\,p}_0(\Curl; \Omega,\R^{3\times3}) \coloneqq \{P\in  W^{1,\,p}(\Curl; \Omega,\R^{3\times3}) \mid P \times \nu = 0 \text{ on } \partial \Omega\},
\end{align*}
where $\nu$ denotes the outward unit normal vector field to $\partial\Omega$,
and the tangential trace $P\times \nu$ is understood in the sense of $W^{-\frac1p,\, p}(\partial \Omega,\R^{3\times3})$ which is justified by partial integration, so that its trace is defined by
\begin{align}\label{eq:partInt}
 \forall\ Q\in&\  W^{1-\frac{1}{p'},\,p'}(\partial\Omega,\R^{3\times 3}):\quad \skalarProd{P\times (-\nu)}{Q}_{\partial \Omega}=  \int_{\Omega}\skalarProd{\Curl P}{\widetilde{Q}}-\skalarProd{P}{\Curl \widetilde{Q}}\, \intd{x},
\end{align}
where $\widetilde{Q}\in W^{1,\,p'}(\Omega,\R^{3\times3})$ denotes any extension of $Q$ in $\Omega$. Here, $\skalarProd{.}{.}_{\partial\Omega}$ indicates the duality pairing between $W^{-\frac1p,\,p}(\partial\Omega,\R^{3\times3})$ and $W^{1-\frac{1}{p'},\,p'}(\partial\Omega,\R^{3\times 3})$.

However, the appearance of the operator ~ $\dev\Curl$ ~ on the right hand side of our designated results in this paper would suggest to work in
\begin{align}
  W^{1,\,p}(\dev\Curl; \Omega,\R^{3\times3}) &\coloneqq \{P\in L^p(\Omega,\R^{3\times3})\mid \dev\Curl P \in L^p(\Omega,\R^{3\times3})\}
\end{align}
but this is, surprisingly at first glance, not a new space:
\begin{lemma}
 ~ $W^{1,\,p}(\dev\Curl; \Omega,\R^{3\times3}) = W^{1,\,p}(\Curl; \Omega,\R^{3\times3})$.
\end{lemma}
 It is sufficient to show that the $p$-integrability of $\dev\Curl P$ already implies the $p$-integrability of $\Curl P$, and follows from the general case:
\begin{lemma}
 Let $P\in\mathscr{D}'(\Omega,\R^{3\times3})$. Then we have for all $m\in\Z$ that
 \begin{equation}
  \Curl P\in W^{m,\,p}(\Omega,\R^{3\times3})\quad \Leftrightarrow\quad \dev\Curl P\in W^{m,\,p}(\Omega,\R^{3\times3}).
 \end{equation}
\end{lemma}

\begin{proof}
 We again consider the decomposition of $P$ into its symmetric and skew-symmetric part, i.e.
 \[
P =  S + A = S + \Anti(a) \qquad \text{for some $S\in\Sym(3)$, $A\in\so(3)$ and with $a=\axl(A)$.}
 \]
Then by Nye's formula \eqref{eq:Nye-prelim_a} we have
\begin{equation}\label{eq:Curl}
 \Curl P = \Curl (S+\Anti(a))\overset{\eqref{eq:Nye-prelim}}{=}\Curl S + \div a\cdot \id - (\D a)^T
\end{equation}
and in view of ~ $\tr(\Curl S)=0$ ~ we obtain
\begin{equation}\label{eq:devCurl}
 \dev\Curl P = \Curl S - (\D a)^T + \frac13\div a\cdot \id
\end{equation}
so that taking the $\Curl$ of the transpositions on both sides gives
\begin{equation}
 \Curl([\dev\Curl P]^T) \underset{\eqref{eq:skalarID}}{\overset{\Curl \circ \D \,\equiv 0}{=} } \underset{\in\Sym(3)}{\underbrace{\inc S}} -\frac13 \underset{\in\so(3)}{\underbrace{\Anti(\nabla \div a)}},
\end{equation}
which gives
\begin{equation}\label{eq:Zerlegung_curldecurl}
 \skew \Curl([\dev\Curl P]^T) = -\frac13\Anti(\nabla \div a).
 \end{equation}
Thus, $\dev\Curl P\in  W^{m,\,p}(\Omega,\R^{3\times3})$ implies $\Curl([\dev\Curl P]^T) \in W^{m-1,\,p}(\Omega, \R^{3\times3})$ as well as
\begin{align}
 \skew \Curl([\dev\Curl P]^T)&=\frac12(\Curl([\dev\Curl P]^T)-[\Curl([\dev\Curl P]^T)]^T)\notag\\
 &\in W^{m-1,\,p}(\Omega, \R^{3\times3}),
\end{align}
so that we obtain
\begin{equation}\label{eq:expression_abl_div_axl_skew}
 \nabla \div a \overset{\eqref{eq:Zerlegung_curldecurl}}{=} -3 \axl \skew \Curl([\dev\Curl P]^T) \in W^{m-1,\,p}(\Omega, \R^3).
\end{equation}
Since $a=\axl\skew P\in\mathscr{D}'(\Omega,\R^3)$, we apply Theorem \ref{th:necas_estimate} to $\div a \in\mathscr{D}'(\Omega,\R)$ to conclude from \eqref{eq:expression_abl_div_axl_skew} that  ~ $\div a\in  W^{m,\,p}(\Omega,\R)$. The statement of the Lemma then follows from the decompositions \eqref{eq:Curl} and \eqref{eq:devCurl} which give the expression
 \begin{equation}\label{eq:decompCurldurchdev}
     \Curl P = \dev\Curl P + \frac23\div a\cdot\id \in  W^{m,\,p}(\Omega,\R^{3\times3}).\qedhere
 \end{equation}
\end{proof}

\begin{corollary}
 The classical Hilbert space ~ $H(\Curl;\Omega,\R^{3\times3})$  ~ coincides with the Hilbert space\linebreak $H(\dev \Curl;\Omega,\R^{3\times3})\coloneqq \{P\in L^2(\Omega,\R^{3\times3})\mid \dev\Curl P \in L^2(\Omega,\R^{3\times3})\}.$
\end{corollary}

{
\begin{remark}[Equivalence of norms] In view of \eqref{eq:expression_abl_div_axl_skew} an application of the Lions lemma to $\div a$, with $a=\axl\skew P$, gives us $\div a \in W^{m,\,p}(\Omega,\R)$. Moreover, by the Ne\v{c}as estimate we have
\begin{align}
\norm{\div a}&_{W^{m,\,p}(\Omega,\R)} \le c_1\,(\norm{\div a}_{W^{m-1,\,p}(\Omega,\R)}+ \norm{\nabla \div a}_{W^{m-1,\,p}(\Omega,\R^3)})\notag\\
& \stackrel{\mathclap{(3.10)}}{\le} \ c_2 \,(\norm{\div \axl\skew P}_{W^{m-1,\,p}(\Omega,\R)}+ \norm{\Curl([\dev\Curl P]^T)}_{W^{m-1,\,p}(\Omega,\R^{3\times 3})})\notag\\[1ex]
& \le c_3 \,(\norm{P}_{W^{m,\,p}(\Omega,\R^{3\times 3})}+ \norm{\dev\Curl P}_{W^{m,\,p}(\Omega,\R^{3\times 3})}),
\end{align}
provided that $P\in W^{m,\,p}(\Omega,\R^{3\times 3})$. Together with \eqref{eq:decompCurldurchdev} we conclude:
\begin{equation}
\norm{\Curl P}_{W^{m,\,p}(\Omega,\R^{3\times 3})} \le c_4\, (\norm{P}_{W^{m,\,p}(\Omega,\R^{3\times 3})}+ \norm{\dev\Curl P}_{W^{m,\,p}(\Omega,\R^{3\times 3})})
\end{equation}
as well as
\begin{align}
\norm{P}_{W^{m,\,p}(\Omega,\R^{3\times 3})} + &\norm{\Curl P}_{W^{m,\,p}(\Omega,\R^{3\times 3})}
\le c_5\, (\norm{P}_{W^{m,\,p}(\Omega,\R^{3\times 3})}+ \norm{\dev\Curl P}_{W^{m,\,p}(\Omega,\R^{3\times 3})})
\end{align}
and especially for $m=0$:
\begin{equation}\label{eq:equivNorm}
\norm{P}_{L^p(\Omega,\R^{3\times 3})} + \norm{\Curl P}_{L^p(\Omega,\R^{3\times 3})} \le c_5\, (\norm{P}_{L^p(\Omega,\R^{3\times 3})}+ \norm{\dev\Curl P}_{L^p(\Omega,\R^{3\times 3})})
\end{equation}
for all $P\in W^{1,\,p}(\Curl;\Omega,\R^{3\times3})$.
\footnote{
This result also follows from the open mapping theorem (also known as Banach-Schauder theorem \cite[Thm 5.6-1]{Ciarlet2013FAbook}) in functional analysis. More precisely, the latter provides the following sufficient condition for two norms to be equivalent in an infinite-dimensional space, see \cite[Thm 5.6-4]{Ciarlet2013FAbook}: 
\begin{corollary}
 Let $\norm{.}$ and $\norm{.}'$ be two norms on the same vector space $X$, with the following properties: both spaces $(X,\norm{.})$ and $(X,\norm{.}')$ are complete, and there exists a constant $C$ such that\[ \norm{x}'\le C\,\norm{x} \quad \text{for all $x\in X$.} \]
 Then the two norms $\norm{.}$ and $\norm{.}'$ are equivalent.
 \end{corollary}
 }
\end{remark}
}

\begin{remark}
 The last identity in \eqref{eq:decompCurldurchdev}, which could also be formally obtained from \eqref{eq:firststepdevPtimesb} with $b=-\nabla$, together with the expression \eqref{eq:expression_abl_div_axl_skew} gives for general matrix field $P\in\mathscr{D}'(\Omega,\R^{3\times3})$:
\begin{equation}\label{eq:lincombi_DCP}
 \D \Curl P = L(\D\, \dev \Curl P).
\end{equation}
Thus, recalling \eqref{eq:lincombi}, we arrive directly at the case \ref{lin_combi_trace_b} of Lemma \ref{lem:lin_combi}.
\end{remark}

\begin{corollary}
Notably, the trace condition in $W^{1,\,p}_ {0}(\dev\Curl; \Omega,\R^{3\times3})$ would read $\dev (P\times \nu) = 0$ on $\partial \Omega$, to be understood by partial integration via
\begin{align}\label{eq:partIntdev}
 \forall\ Q\in  W^{1-\frac{1}{p'},\,p'}(\partial\Omega,\R^{3\times 3}):\quad \skalarProd{\dev(P\times (-\nu))}{Q}_{\partial \Omega} &=  \int_{\Omega}\skalarProd{\dev\Curl P}{\widetilde{Q}}-\skalarProd{P}{\Curl\dev \widetilde{Q}}\, \intd{x}\\
 & = \int_{\Omega}\skalarProd{\Curl P}{\dev\widetilde{Q}}-\skalarProd{P}{\Curl\dev \widetilde{Q}}\, \intd{x} \notag\\
  & \overset{\eqref{eq:partInt}}{=} \skalarProd{P\times (-\nu)}{\dev Q}_{\partial \Omega},\notag
\end{align}
where $\widetilde{Q}\in W^{1,\,p'}(\Omega,\R^{3\times3})$ denotes any extension of $Q$ in $\Omega$. However, it follows from Observation \ref{obs:traces_equal} that the boundary conditions $P\times \nu = 0$ and $\dev(P\times \nu)=0$ on $\partial\Omega$ are the same.
\end{corollary}

\begin{lemma}\label{lem:basic2}
  Let $\Omega \subset \R^3$ be a bounded Lipschitz domain, $1<p<\infty$ and $P\in\mathscr{D}'(\Omega,\R^{3\times3})$. Then either of the conditions
  \begin{thmenum}
  \item $\dev\sym P\in L^p(\Omega,\R^{3\times3})$ and $\Curl P \in W^{-1,\,p}(\Omega,\R^{3\times3})$, \label{condition_a_dev}
  \item $\sym P\in L^p(\Omega,\R^{3\times3})$ and $\dev\Curl P \in W^{-1,\,p}(\Omega,\R^{3\times3})$, \label{condition_b_dev}
  \item $\dev\sym P\in L^p(\Omega,\R^{3\times3})$ and $\dev\Curl P \in W^{-1,\,p}(\Omega,\R^{3\times3})$,\label{condition_c_dev}
  \end{thmenum}
implies $P\in L^p(\Omega,\R^{3\times3})$. Moreover, we have the corresponding estimates
\begin{subequations}
  \begin{align}
  \norm{P}_{L^p(\Omega,\R^{3\times3})} &\leq c\, \Big(\norm{\skew P+\textstyle\frac13\tr P\cdot \id}_{W^{-1,\,p}(\Omega,\R^{3\times3})}\notag\\
  &\hspace{4em}+ \norm{\dev\sym P}_{L^p(\Omega,\R^{3\times3})}+ \norm{ \Curl P }_{W^{-1,\,p}(\Omega,\R^{3\times3})}\Big),\label{eq:from_a_dev}\\
  \norm{P}_{L^p(\Omega,\R^{3\times3})} &\leq c\, \Big(\norm{\skew P}_{W^{-1,\,p}(\Omega,\R^{3\times3})}\notag\\
  &\hspace{4em}+ \norm{\sym P}_{L^p(\Omega,\R^{3\times3})}+ \norm{\dev \Curl P }_{W^{-1,\,p}(\Omega,\R^{3\times3})}\Big),\label{eq:from_b_dev}\\
   \norm{P}_{L^p(\Omega,\R^{3\times3})} &\leq c\, \Big(\norm{\skew P+\textstyle\frac13\tr P\cdot \id}_{W^{-1,\,p}(\Omega,\R^{3\times3})}\notag\\
  &\hspace{4em}+ \norm{\dev\sym P}_{L^p(\Omega,\R^{3\times3})}+ \norm{\dev \Curl P }_{W^{-1,\,p}(\Omega,\R^{3\times3})}\Big),\label{eq:from_c_dev}
  \end{align}
 \end{subequations}
  each with a constant $c=c(p,\Omega)>0$.
\end{lemma}

 \begin{proof}
 We start by proving part \ref{condition_b_dev}. For that purpose we will follow the proof of \cite[Lemma 3.1]{agn_lewintan2019KornLp}. Thus, for part \ref{condition_b_dev} it remains to deduce that $\skew P\in L^p(\Omega,\R^{3\times3})$. We have
 \begin{align}
  \norm{\D^2\skew P}_{W^{-2,\,p}(\Omega,\R^{ {3\times3^3}})}\quad&\overset{\mathclap{\text{Lem. \ref{lem:lin_combi}\ref{lin_combi_trace_b}}}}{\le}\quad c\,   \norm{\D \dev \Curl\skew P}_{W^{-2,\,p}(\Omega,\R^{ {3\times3^2}})}\notag\\
  &\le c\, \norm{\dev \Curl(P-\sym P)}_{W^{-1,\,p}(\Omega,\R^{3\times3})}\notag\\
  & \le c\,( \norm{\dev \Curl P }_{W^{-1,\,p}(\Omega,\R^{3\times3})} + \norm{\Curl \sym P}_{W^{-1,\,p}(\Omega,\R^{3\times3})})\notag\\
  & \le c\,( \norm{\sym P}_{L^p(\Omega,\R^{3\times3})}+ \norm{\dev \Curl P }_{W^{-1,\,p}(\Omega,\R^{3\times3})}).\label{eq:d2skew_hilfs}
 \end{align}
Hence, the assumptions of part \ref{condition_b_dev} yield $\D^2\skew P\in W^{-2,\,p}(\Omega,\R^{ {3\times3^3}})$, so that, by Corollary \ref{cor:LionsNecas_k}, we obtain $\skew P\in L^p(\Omega,\R^{3\times3})$ and moreover the estimate
\begin{align}
 \norm{\skew P}&_{L^p(\Omega,\R^{3\times3})} \leq c\, (\norm{\skew P}_{W^{-1,\,p}(\Omega,\R^{3\times3})}+\norm{\D^2\skew P}_{W^{-2,\,p}(\Omega,\R^{ {3\times3^3}})})\notag\\
 & \overset{\eqref{eq:d2skew_hilfs}}{\le}c\, \Big(\norm{\skew P}_{W^{-1,\,p}(\Omega,\R^{3\times3})}
 + \norm{\sym P}_{L^p(\Omega,\R^{3\times3})}+ \norm{\dev \Curl P }_{W^{-1,\,p}(\Omega,\R^{3\times3})}\Big).
\end{align}
Then by adding $\norm{\sym P}_{L^p(\Omega,\R^{3\times3})}$ on both sides we obtain \eqref{eq:from_b_dev}.\medskip

 Clearly, the conclusion of \ref{condition_a_dev} as well as the estimate \eqref{eq:from_a_dev} follow from \ref{condition_c_dev} and \eqref{eq:from_c_dev}, respectively. To establish \ref{condition_c_dev}, we make use of the orthogonal decomposition {$P = \dev\sym P + (\skew P + \textstyle\frac13\tr P \cdot \id)$.}
Then, to obtain $\skew P + \textstyle\frac13\tr P \cdot \id\in L^p(\Omega,\R^{3\times3})$ for \ref{condition_c_dev}, we consider
\begin{align}\label{eq:D2curlskew+tr}
&\hspace{-3em}\norm{\D^2\dev\Curl (\skew P+ \textstyle\frac13\tr P \cdot \id)}_{ W^{-3,\,p}(\Omega,\R^{3\times3^3})} \notag\\
&\leq c\, \norm{\dev\Curl (P-\dev\sym P)}_{ W^{-1,\,p}(\Omega,\R^{3\times3})}\notag\\
&\le c\,(\norm{\dev\Curl P}_{ W^{-1,\,p}(\Omega,\R^{3\times3})} + \norm{\Curl \dev\sym P}_{ W^{-1,\,p}(\Omega,\R^{3\times3})}) \notag \\
&\leq c \,(\norm{\dev\Curl P}_{ W^{-1,\,p}(\Omega,\R^{3\times3})} + \norm{\dev\sym P}_{ L^p(\Omega,\R^{3\times3})}).
\end{align}
Therefore, $\D^2\dev\Curl (\skew P+ \textstyle\frac13\tr P \cdot \id)\in W^{-3,\,p}(\Omega,\R^{3\times3^3})$ \ follows from the assumptions of \ref{condition_c_dev} and Lemma \ref{lem:lin_combi} \ref{lin_combi_trace_c}  implies
\begin{equation}
\D^3 (\skew P+ \textstyle\frac13\tr P \cdot \id)\in W^{-3,\,p}(\Omega,\R^{3\times3^4})\,.
\end{equation}
Applying Corollary \ref{cor:LionsNecas_k} again, this time to $\skew P+ \textstyle\frac13\tr P \cdot \id$, we arrive at $\skew P + \textstyle\frac13\tr P \cdot \id\in L^p(\Omega,\R^{3\times3})$ and, moreover,
\begin{align}\label{eq:Norm_c}
 \norm{\skew P+\textstyle\frac13\tr P \cdot \id}_{L^p(\Omega,\R^{3\times 3})}&\le c\,
 \Big(\norm{\skew P+\textstyle\frac13\tr P\cdot \id}_{W^{-1,\,p}(\Omega,\R^{3\times3})}\notag\\
  &\qquad\ +\norm{\D^3(\skew P+\textstyle\frac13\tr P\cdot \id)}_{W^{-3,\,p}(\Omega,\R^{3\times3^4})} \Big)\notag\\
  &\overset{\mathclap{\text{Lem. \ref{lem:lin_combi} \ref{lin_combi_trace_c}}}}{\leq}\quad  c\,\Big(\norm{\skew P+\textstyle\frac13\tr P\cdot \id}_{W^{-1,\,p}(\Omega,\R^{3\times3})}\\
  &\quad\qquad\ +\norm{\D^2\dev\Curl(\skew P+\textstyle\frac13\tr P\cdot \id)}_{W^{-3,\,p}(\Omega,\R^{3\times3^3})} \Big)\notag\\
  &\overset{\mathclap{\eqref{eq:D2curlskew+tr}}}{\leq} \ c\,
 \Big(\norm{\skew P+\textstyle\frac13\tr P\cdot \id}_{W^{-1,\,p}(\Omega,\R^{3\times3})}\notag\\
  &\ \qquad\ +\norm{\dev\sym P}_{L^p(\Omega,\R^{3\times3})}+ \norm{ \dev\Curl P }_{W^{-1,\,p}(\Omega,\R^{3\times3})}\Big).\notag\qedhere
\end{align}
 \end{proof}
 \begin{remark}
 Of course, part \ref{condition_a_dev} can also be proven independently of part \ref{condition_c_dev}. Indeed, using Lemma \ref{lem:lin_combi} \ref{lin_combi_trace_a} we obtain
 \begin{align}\label{eq:curlskew+tr}
 \norm{\D^2(\skew P+\textstyle\frac13\tr P\cdot \id)}_{W^{-2,\,p}(\Omega,\R^{ {3\times3^3}})}
 &\overset{\mathclap{\text{Lem. \ref{lem:lin_combi} \ref{lin_combi_trace_a}}}}{\leq}\quad  c\,
\norm{\D\Curl (\skew P+ \textstyle\frac13\tr P \cdot \id)}_{ W^{-2,\,p}(\Omega,\R^{ {3\times3^2}})}\notag\\
& \leq c\, \norm{\Curl(P-\dev\sym P)}_{ W^{-1,\,p}(\Omega,\R^{3\times3})} \notag \\
&\leq c \,(\norm{\Curl P}_{ W^{-1,\,p}(\Omega,\R^{3\times3})} + \norm{\dev\sym P}_{ L^p(\Omega,\R^{3\times3})})
\end{align}
and the conclusion follows from an application of Corollary \ref{cor:LionsNecas_k} to $\skew P+ \textstyle\frac13\tr P \cdot \id$.
 \end{remark}

 The rigidity results now follow by elimination of the corresponding first term on the right-hand side.

\begin{theorem}\label{thm:main1tf}
 Let $\Omega \subset \R^3$ be a bounded Lipschitz domain and $1<p<\infty$. There exists a constant $c=c(p,\Omega)>0$ such that for all $P\in { L^p}(\Omega,\R^{3\times3})$ we have
  \begin{subequations}
 \begin{align}
   \inf_{T\in K_{dS,C}}\norm{P-T}_{L^p(\Omega,\R^{3\times3})}
   &\leq c\,\left(\norm{\dev \sym P }_{L^p(\Omega,\R^{3\times3})}+ \norm{ \Curl P }_{W^{-1,\,p}(\Omega,\R^{3\times3})}\right)\,,\label{eq:rigid_Dev_a}\\
   \inf_{T\in K_{S,dC}}\norm{P-T}_{L^p(\Omega,\R^{3\times3})}
   &\leq c\,\left(\norm{\sym P }_{L^p(\Omega,\R^{3\times3})}+ \norm{\dev \Curl P }_{W^{-1,\,p}(\Omega,\R^{3\times3})}\right)\,,\label{eq:rigid_Dev_b}\\
   \inf_{T\in K_{dS,dC}}\norm{P-T}_{L^p(\Omega,\R^{3\times3})}
   &\leq c\,\left(\norm{\dev \sym P }_{L^p(\Omega,\R^{3\times3})}+ \norm{\dev \Curl P }_{W^{-1,\,p}(\Omega,\R^{3\times3})}\right)\,,\label{eq:rigid_Dev_c}
 \end{align}
\end{subequations}
 where the kernels are given, respectively, by
 \begin{subequations}
 \begin{align}
  K_{dS,C}&=\{T:\Omega\to\R^{3\times3}\mid T(x)=\Anti(\widetilde{A}\,x+b)+(\skalarProd{\axl \widetilde{A}}{x}+\beta) {\cdot}\id, 
  \widetilde{A}\in\so(3), b\in\R^3, \beta\in\R \} {,}\label{eq:kernel_dSC_thm}\\
   K_{S,dC} &= \{T:\Omega\to\R^{3\times3}\mid  T(x)=\Anti(\beta\,x+b), \ b\in\R^3, \beta\in\R\}, \label{eq:kernel_SdC_thm}\\
    K_{dS,dC} &= \{T:\Omega\to\R^{3\times3}\mid  T(x)=\Anti\big(\widetilde{A}\,x+\beta\, x+b \big)+\big(\skalarProd{\axl\widetilde{A}}{x}+\gamma \big) {\cdot}\id, \notag \\ &\hspace{10em} \widetilde{A}\in\so(3), b\in\R^3, \beta,\gamma\in\R\}.\label{eq:kernel_dSdC_thm}
 \end{align}
\end{subequations}
\end{theorem}

\begin{proof}
 We proceed as in the proof of Korn's inequalities \eqref{eq:Korn_Lp_w} resp.\ \eqref{eq:KornQuant}, see \cite[Theorem 3.3]{agn_lewintan2019KornLp} resp.\ \cite[Theorem 6.15-3]{Ciarlet2013FAbook}, and start by characterizing the kernel of the right-hand side,
 \begin{align*}
   K_{dS,C} \coloneqq \{P\in { L^p}(\Omega,\R^{3\times3}) \mid \ & \dev \sym P = 0 \text{ a.e. and }\\&  \Curl P= 0 \text{ in the distributional sense}\},
 \end{align*}
so that $P\in K_{dS,C}$ if and only if ~ $P=\skew P + \frac13\tr P\cdot\id$ ~ and ~ $\Curl(\skew P + \frac13\tr P\cdot\id)\equiv 0$. Hence, \eqref{eq:kernel_dSC_thm} follows by virtue of Lemma \ref{lem:Kernel} \ref{kernel_a}.

Let us denote by $e_1,\ldots,e_M$ a basis of $K_{dS,C}$, where $M\coloneqq \dim  K_{dS,C} = 7$, and by $\ell_1, \ldots, \ell_M$ the corresponding continuous linear forms on $K_{dS,C}$ given by
\begin{equation}\label{eq:basisKern}
 \ell_\alpha(e_j)\coloneqq \delta_{\alpha j}.
\end{equation}
By the Hahn-Banach theorem in a normed vector space (see e.g. \cite[Theorem 5.9-1]{Ciarlet2013FAbook}), we extend $\ell_\alpha$ to continuous linear forms - again denoted by $\ell_\alpha$ - on the Banach space ${ L^p}(\Omega,\R^{3\times3})$, $1\le\alpha\le M$. Notably,
\[
 T\in K_{dS,C} \text{ is equal to $0$ } \quad \Leftrightarrow \quad \ell_\alpha(T)= 0 \ \forall\ \alpha\in\{1,\ldots,M\}.
\]
Following the proof of \cite[Theorem 3.4]{agn_lewintan2019KornLp} we eliminate the first term on the right-hand side of \eqref{eq:from_a_dev} by exploiting the compactness $L^p(\Omega,\R^{3\times 3})\subset\!\subset  W^{-1,\,p}(\Omega,\R^{3\times 3})$  and arrive at
\begin{align}\label{eq:hilfsungl}
 \norm{P}_{L^p(\Omega,\R^{3\times 3})}\leq c\,\left(\norm{\dev \sym P }_{L^p(\Omega,\R^{3\times 3})}+ \norm{ \Curl P }_{W^{-1,\,p}(\Omega,\R^{3\times 3})}+\sum_{\alpha=1}^M\abs{\ell_\alpha(P)} \right).
\end{align}
 {Indeed, if \eqref{eq:hilfsungl} were false, there would exist a sequence $P_k\in L^p(\Omega,\R^{3\times3})$ such that
\begin{align*}
 \norm{P_k}_{L^p(\Omega,\R^{3\times3})}=1
 \quad \text{and}\quad \left(\norm{\dev \sym P_k }_{L^p(\Omega,\R^{3\times 3})}+ \norm{ \Curl P_k }_{W^{-1,\,p}(\Omega,\R^{3\times 3})}+\sum_{\alpha=1}^M\abs{\ell_\alpha(P_k)} \right)<\frac1k.
\end{align*}
Thus, for a subsequence $P_k\rightharpoonup P^*$ in $L^p(\Omega,\R^{3\times3}$ with $\dev\sym P^*=0$ a.e., $\sym\Curl P^*=0$ in the distributional sense and $\ell_\alpha(P_k)=0$ for all $\alpha=1,\ldots,M$, so that $P^*=0$ a.e.. By the compact embedding $L^p(\Omega,\R^{3\times 3})\subset\!\subset  W^{-1,\,p}(\Omega,\R^{3\times 3})$ there exists a subsequence $P_k$, so that ~ $\skew P_k+\frac13\tr P_k\cdot\id\to0$~ in $ W^{-1,\,p}(\Omega,\R^{3\times 3})$. This is a contradiction to \eqref{eq:from_a_dev}.

}
Considering now the projection ~ $\pi_{a}:{ L^p}(\Omega,\R^{3\times 3}) \to K_{dS,C}$ ~ given by
\begin{equation}\label{eq:projection}
 \pi_{a}(P)\coloneqq \sum_{j=1}^M \ell_j(P)\, e_j
\end{equation}
we obtain $\ell_\alpha(P-\pi_{a}(P)) \overset{\eqref{eq:basisKern}}{=} 0$ for all $1\le\alpha\le M$, so that \eqref{eq:rigid_Dev_a} follows after applying \eqref{eq:hilfsungl} to $P-\pi_{a}(P)$.\medskip

Furthermore, we obtain the characterizations \eqref{eq:kernel_SdC_thm} and \eqref{eq:kernel_dSdC_thm} by Lemma \ref{lem:Kernel} \ref{kernel_b} and \ref{kernel_c}, respectively, since
\begin{align}
  K_{S,dC} &\coloneqq \{P\in  { L^p}(\Omega,\R^{3\times3}) \mid  \sym P = 0 \text{ a.e. and }  \dev\Curl P= 0 \text{ in the distributional sense}\}\\
  &\overset{\mathclap{\text{Lemma \ref{lem:Kernel} \ref{kernel_b}}}}{=}\hspace{2.5em}\{T:\Omega\to\R^{3\times3}\mid  T(x)=\Anti(\beta\,x+b), \ b\in\R^3, \beta\in\R\} \notag\\
  \shortintertext{and}
    K_{dS,dC} &\coloneqq \{P\in  { L^p}(\Omega,\R^{3\times3}) \mid  \dev\sym P= 0 \text{ a.e. and }   \dev\Curl P=  0 \text{ in the distributional sense}\}\\
  &\overset{\mathclap{\text{Lemma \ref{lem:Kernel} \ref{kernel_c}}}}{=}\hspace{2em}\{T:\Omega\to\R^{3\times3}\mid  T(x)=\Anti\big(\widetilde{A}\,x+\beta\, x+b \big)+\big(\skalarProd{\axl\widetilde{A}}{x}+\gamma \big) {\cdot}\id,\notag\\
  &\hspace{21em} \widetilde{A}\in\so(3), b\in\R^3, \beta,\gamma\in\R\}\notag
\end{align}
with $\dim K_{S,dC}=4$ ~ and ~ $\dim K_{dS,dC}=8$. Hence, we can argue as above to deduce \eqref{eq:rigid_Dev_b} and  \eqref{eq:rigid_Dev_c}  from \eqref{eq:from_b_dev} and \eqref{eq:from_c_dev}, respectively, since we end up with
\begin{align}
 \norm{P-\pi_{b}(P)}_{L^p(\Omega,\R^{3\times3})}&\leq c\,\left(\norm{\sym P }_{L^p(\Omega,\R^{3\times3})}+ \norm{\dev \Curl P }_{W^{-1,\,p}(\Omega,\R^{3\times3})}\right)\\
 \shortintertext{and}
\norm{P-\pi_{c}(P)}_{L^p(\Omega,\R^{3\times3})}&\leq c\,\left(\norm{\dev \sym P }_{L^p(\Omega,\R^{3\times3})}+ \norm{\dev \Curl P }_{W^{-1,\,p}(\Omega,\R^{3\times3})}\right)
\end{align}
respectively, with projections ~ $\pi_{b}:{ L^p}(\Omega,\R^{3\times 3}) \to K_{S,dC}$ ~ and  ~ $\pi_{c}:{ L^p}(\Omega,\R^{3\times 3}) \to K_{dS,dC}$.
\end{proof}

Finally, the kernel is killed by the tangential trace condition $P\times \nu \equiv 0$ ~ ($\Leftrightarrow~\dev(P\times \nu)=0$, cf. Obs. \ref{obs:traces_equal}):

\begin{theorem} \label{thm:main2tracefree}
Let $\Omega \subset \R^3$ be a bounded Lipschitz domain and $1<p<\infty$. There exists a constant $c=c(p,\Omega)>0$ such that for all $P\in  W^{1,\,p}_0(\Curl; \Omega,\R^{3\times3})$ we have
 \begin{equation}
     \norm{ P }_{L^p(\Omega,\R^{3\times3})}\leq c\,\left(\norm{\dev \sym P }_{L^p(\Omega,\R^{3\times3})}+ \norm{ \dev\Curl P }_{L^p(\Omega,\R^{3\times3})}\right).\label{eq:Korn_Lp_thm_dSdC}
 \end{equation}
\end{theorem}
\begin{proof}
We argue as in the proof of \cite[Theorem 3.5]{agn_lewintan2019KornLp} and consider a sequence $\{P_k\}_{k\in\N}\subset W^{1,\,p}_0(\Curl;\Omega,\R^{3\times3})$ which converges weakly in $L^p(\Omega,\R^{3\times3})$ to $P^*$ so that \ $\dev\sym P^* = 0$ \ a.e.~and \ $\dev\Curl P^*= 0$ in the distributional sense, \ i.e.\ $P^*\in K_{dS,dC}$, where
\begin{align*}
 K_{dS,dC} &\overset{\eqref{eq:kernel_dSdC_thm}}= \{T:\Omega\to\R^{3\times3}\mid  T(x)=\Anti\big(\widetilde{A}\,x+\beta\, x+b \big)+\big(\skalarProd{\axl\widetilde{A}}{x}+\gamma \big) {\cdot}\id, \notag \\ &\hspace{20em} \widetilde{A}\in\so(3), b\in\R^3, \beta,\gamma\in\R\}.
\end{align*}
By \eqref{eq:partIntdev} it further follows that $\skalarProd{\dev(P^*\times (-\nu))}{Q}_{\partial \Omega}=0$ for all $Q\in  W^{1-\frac{1}{p'},\,p'}(\partial\Omega,\R^{3\times 3})$. However, since $P^*\in K_{dS,dC}$ also has an explicit representation, the boundary condition $\dev(P^*\times\nu)=0$ is also valid in the classical sense. Furthermore, we deduce by Observation \ref{obs:traces_equal} that $P^*\times \nu=0$ on $\partial \Omega$, so that  $P^*\in W^{1,\,p}_0(\Curl; \Omega,\R^{3\times3})$. Again, using the explicit representation of $P^*=\Anti\big(\widetilde{A}\,x+\beta\, x +b\big)+\big(\skalarProd{\axl\widetilde{A}}{x}+\gamma \big) {\cdot}\id$, we conclude with Observation \ref{obs:zwei} that, in fact, $P^*\equiv 0$:
\begin{align*}
 [\Anti\big(\widetilde{A}\,x&+\beta\, x+b \big)+\big(\skalarProd{\axl\widetilde{A}}{x}+\gamma \big) {\cdot}\id] \times \nu = 0 \\
 &\overset{\text{Obs. \ref{obs:zwei}}}{\Rightarrow} \quad \widetilde{A}\,x+\beta\, x+b =0 \quad \text{and} \quad \skalarProd{\axl\widetilde{A}}{x}+\gamma=0 \quad \text{for all $x\in\partial\Omega$}\\
 &\quad \Rightarrow\quad \gamma = 0, \ \widetilde{A}= 0 \quad \Rightarrow \quad b=0, \beta=0.\qedhere
\end{align*}

\end{proof}

\begin{remark}
 Similarly, the following estimates can also be deduced, even independently of \eqref{eq:Korn_Lp_thm_dSdC}, for $P\in  W^{1,\,p}_0(\Curl; \Omega,\R^{3\times3})$:
  \begin{align}
 \norm{ P }_{L^p(\Omega,\R^{3\times3})}&\leq c\,\left(\norm{\dev \sym P }_{L^p(\Omega,\R^{3\times3})}+ \norm{ \Curl P }_{L^p(\Omega,\R^{3\times3})}\right),\label{eq:Korn_Lp_thm_dSC}\\
 \norm{ P }_{L^p(\Omega,\R^{3\times3})}&\leq c\,\left(\norm{\sym P }_{L^p(\Omega,\R^{3\times3})}+ \norm{ \dev\Curl P }_{L^p(\Omega,\R^{3\times3})}\right).\label{eq:Korn_Lp_thm_SdC}
 \end{align}
Since by \cite[Theorem 3.1 (ii)]{agn_bauer2013dev} it holds
\begin{equation}\label{eq:traceCurl}
 \norm{\Curl P}_{L^p(\Omega,\R^{3\times3})}\leq c\, \norm{\dev \Curl P }_{L^p(\Omega,\R^{3\times3})} \qquad \text{for } P\in W^{1,\,p}_0(\Curl;\Omega,\R^{3\times3}),
\end{equation}
we can recover \eqref{eq:Korn_Lp_thm_dSdC} from \eqref{eq:Korn_Lp_thm_dSC} and \eqref{eq:traceCurl}.\medskip

However, without boundary conditions the Ne\v{c}as estimate provides for $ P\in  W^{1,\,p}(\Curl; \Omega,\R^{3\times3})$:
\begin{align}
 \norm{\Curl P}_{L^p(\Omega,\R^{3\times3})} &\overset{\eqref{eq:necas_m_p}}{\le} c\, (\norm{\Curl P}_{W^{-1,\,p}(\Omega,\R^{3\times3})}+\norm{\D \Curl P}_{W^{-1,\,p}(\Omega,\R^{ {3\times3^2}})}) \notag\\
 &\overset{\eqref{eq:lincombi_DCP}}{\le} c\, (\norm{\Curl P}_{W^{-1,\,p}(\Omega,\R^{3\times3})}+\norm{\D \dev\Curl P}_{W^{-1,\,p}(\Omega,\R^{ {3\times3^2}})}) \notag\\
 &~ \le~  c\, (\norm{\Curl P}_{W^{-1,\,p}(\Omega,\R^{3\times3})}+\norm{\dev\Curl P}_{L^p(\Omega,\R^{3\times3})}).
\end{align}
\end{remark}

\begin{remark}
Among the inequalities \eqref{eq:Korn_Lp_thm_dSdC}, \eqref{eq:Korn_Lp_thm_dSC} and \eqref{eq:Korn_Lp_thm_SdC} we expect \eqref{eq:Korn_Lp_thm_dSC} also to hold true in higher space dimensions $n>3$, see the discussion in our Introduction.
\end{remark}

\begin{remark}
 Regarding {\eqref{eq:equivNorm} and \eqref{eq:Korn_Lp_thm_dSdC} or} \eqref{eq:traceCurl} and \eqref{eq:Korn_Lp_thm_dSdC} we obtain the norm equivalence
 \begin{align*}
\norm{ P }_{L^p(\Omega,\R^{3\times3})}+&\ \norm{ \Curl P }_{L^p(\Omega,\R^{3\times3})}\leq c\,\left(\norm{\dev \sym P }_{L^p(\Omega,\R^{3\times3})} + \norm{ \dev \Curl P }_{L^p(\Omega,\R^{3\times3})}\right)
\end{align*}
for tensor fields $P\in  W^{1,\,p}_0(\Curl; \Omega,\R^{3\times3})$.
\end{remark}

{
For $P=\D u$ in \eqref{eq:Korn_Lp_thm_dSdC} we recover the following tangential trace-free Korn inequality:
\begin{corollary}
 Let $\Omega \subset \R^3$ be a bounded Lipschitz domain and $1<p<\infty$. There exists a constant $c=c(p,\Omega)>0$ such that for all $u\in  W^{1,\,p}(\Omega,\R^3)$ with $\D u \times \nu=0$ on $\partial\Omega$ we have
 \begin{equation}
     \norm{ \D u }_{L^p(\Omega,\R^{3\times 3})}\leq c\,\norm{\dev \sym \D u }_{L^p(\Omega,\R^{3\times3})}.
 \end{equation}
\end{corollary}

For skew-symmetric $P=\Anti(a)$ we recover from \eqref{eq:Korn_Lp_thm_dSdC}  a Poincaré inequality involving only the deviatoric (trace-free) part of the gradient:
\begin{corollary}
 Let $\Omega \subset \R^3$ be a bounded Lipschitz domain and $1<p<\infty$. There exists a constant $c=c(p,\Omega)>0$ such that for all $a\in  W^{1,\,p}_0(\Omega,\R^3)$ we have
 \begin{equation}
     \norm{a}_{L^p(\Omega,\R^{3})}\leq c\,\norm{\dev \D a }_{L^p(\Omega,\R^{3\times3})}.
 \end{equation}
\end{corollary}
\begin{proof}
 This follows from Theorem \ref{thm:main2tracefree} by setting $P=\Anti(a)$ and the following observations:\\
 $\Anti(a)\times \nu =0 \ \Leftrightarrow \ a = 0$ on $\partial\Omega$, see Observation \ref{obs:zwei}, $\Curl (\Anti(a))=L(\D a)$, see \eqref{eq:Nye-prelim_a} and the form of $\Anti(a)$, see \eqref{eq:antiform}. 
\end{proof}

}
\begin{remark}
The previous results also hold true for functions with vanishing  tangential trace only on a relatively open (non-empty) subset $\Gamma \subseteq \partial\Omega$ of the boundary.
 { So, e.g., we have
\begin{equation}
 \norm{ P }_{L^p(\Omega,\R^{3\times3})}\leq c\,\left(\norm{\dev \sym P }_{L^p(\Omega,\R^{3\times3})}+ \norm{ \dev\Curl P }_{L^p(\Omega,\R^{3\times3})}\right)
\end{equation}
for all ~ $P\in W^{1,\,p}_{\Gamma,0}(\Curl;\Omega,\R^{3\times3})$, ~ which is the completion of ~ $C^\infty_{\Gamma,0}(\Omega,\R^{3\times3})$ ~ with respect to the $W^{1,\,p}(\Curl;\Omega,\R^{3\times3})$-norm.
}\end{remark}

\begin{remark}
 In \cite{Garroni10} the authors proved that in $n=2$ dimensions, for $p=2$ a Korn inequality for incompatibile fields also holds true when $\Curl P$ is only in $L^1$ (actually when it is a measure with bounded total variation) under the normalization condition $\int_\Omega\skew P\,\intd{x}=0$.  {In terms of scaling, it is interesting to involve in \eqref{eq:Korn_Lp_thm_dSdC} the Sobolev exponent. So, we will show in a forthcoming paper that for $1<p<3$ the following estimate holds true on an arbitrary open set $\Omega\subseteq\R^3$:
 \begin{equation}
 \norm{P}_{L^{p^*}(\Omega,\R^{3\times3})}\le c\, (\norm{\dev\sym P}_{L^{p^*}(\Omega,\R^{3\times3})}+\norm{\dev\Curl P}_{L^p(\Omega,\R^{3\times3})})
   \end{equation}
for all $P\in C^\infty_c(\Omega,\R^{3\times3})$, where $p^*=\frac{3p}{3-p}$. However, we do not know if such a result  still holds in the borderline case $p=1$.
 }

\end{remark}

\subsubsection*{Acknowledgment}
The authors are grateful for inspiring discussions with Stefan M\"uller (Hausdorff Center for Mathematics, Bonn, Germany)  {and thank also the anonymous referee for his valuable comments and suggestions}. This work was initiated in the framework of the Priority Programme SPP 2256 'Variational Methods for Predicting Complex Phenomena in Engineering Structures and Materials' funded by the Deutsche Forschungsgemeinschaft (DFG, German research foundation), Project-ID 422730790. The second author was supported within the project 'A variational scale-dependent transition scheme - from Cauchy elasticity to the relaxed micromorphic continuum’ (Project-ID 440935806). Moreover, both authors were supported in the Project-ID 415894848 by the Deutsche Forschungsgemeinschaft.

 \printbibliography

 {\footnotesize
 \begin{alphasection}
\section{Appendix}
\subsection{On the trace-free Korn's first inequality in $L^2$}
Using partial integration (see also \cite[Appendix A.1]{agn_neff2015poincare}) we catch up with a simple proof of
\begin{lemma}
 Let $n\ge2$, $\Omega(\text{open})\subset\R^n$, $u\in W^{1,\,2}_0(\Omega, \R^n)$. Then
 \begin{equation}\label{eq:babytrace}
  \int_\Omega \norm{\D u}^2 \intd{x} \leq 2 \int_\Omega \norm{\dev_n\sym \D u}^2\intd{x}.
 \end{equation}
\end{lemma}
\begin{proof}
 For $u\in C^\infty_c(\Omega,\R^n)$ we have
 \begin{align}
  2 \int_\Omega \norm{\sym \D u}^2\intd{x} &= \int_\Omega \norm{\D u}^2 + \sum_{i,j=1}^n(\partial_i u_j)(\partial_j u_i)\intd{x} \overset{\text{part. int.}}{=}\int_\Omega \norm{\D u}^2 + \sum_{i,j=1}^n(\partial_j u_j)(\partial_i u_i)\intd{x}\notag\\
  & = \int_\Omega \norm{\D u}^2 + (\div u)^2\intd{x},\label{eq:Schritt1_baby}
 \end{align}
from where the "baby" Korn inequality ~ $\int_\Omega \norm{\D u}^2 \intd{x} \leq 2 \int_\Omega \norm{\sym \D u}^2\intd{x}$ ~ for $u\in W^{1,\,2}_0(\Omega, \R^n)$ follows. Its improvement is obtained in regard with the decomposition
\begin{equation}\label{eq:decomposition}
 \norm{\dev_n\sym \D u}^2 = \norm{\sym \D u - \frac1n\underset{=\div u}{\underbrace{\tr(\sym\D u)}}\cdot\id}^2 = \norm{\sym \D u}^2 -\frac1n(\div u)^2,
\end{equation}
since we obtain
\begin{align}
 2 \int_\Omega \norm{\dev_n\sym \D u}^2\intd{x}&\overset{\eqref{eq:decomposition}}{=} 2 \int_\Omega \norm{\sym \D u}^2\intd{x} -\frac2n\int_\Omega(\div u)^2\intd{x} \overset{\eqref{eq:Schritt1_baby}}{=} \int_\Omega \norm{\D u}^2\intd{x} +\frac{n-2}{n}\int_\Omega(\div u)^2\intd{x}\notag\\
 &\overset{n\ge2}{\ge}\int_\Omega \norm{\D u}^2\intd{x}.\notag\qedhere
\end{align}
\end{proof}
\begin{remark}
 The trace-free Korn's first inequality \eqref{eq:babytrace} is also valid in $L^p$, $p>1$, see \cite[Prop. 1]{FuchsSchirra2009tracefree2D} for the $n=2$ case and \cite[Thm. 2.3]{Schirra2012tracefreenD} for all $n\geq2$ where again the justification was based on the Lions lemma.
\end{remark}

\subsection{Infinitesimal planar conformal mappings}
Infinitesimal conformal mappings are defined by $\dev_n\sym\D u \equiv 0$ and in $n>2$ they have the representation
\[
\skalarProd{a}{x}\,x-\frac12a\norm{x}^2+A\,x +\beta\,x +c, \qquad \text{with $A\in\so(n)$,  $a,c\in\R^n$ and $\beta\in\R,$}
\]
cf. \cite{Reshetnyak1970,Neff_Jeong_IJSS09,agn_jeong2008existence,Dain2006tracefree,Reshetnyak1994,Schirra2012tracefreenD}.

In the planar case, the situation is quite different. Indeed, the condition $\dev_2\sym \D u \equiv0$ reads
\begin{align*}
 &\begin{pmatrix}
  u_{1,x} & \frac12(u_{1,y}+u_{2,x}) \\ \frac12(u_{1,y}+u_{2,x}) & u_{2,y}
 \end{pmatrix}- \frac12(u_{1,x}+u_{2,y})\cdot\begin{pmatrix}1&0\\0&1\end{pmatrix} = 0 \\
 \Leftrightarrow \quad &
 \begin{pmatrix}
  \frac12(u_{1,x}-u_{2,y}) & \frac12(u_{1,y}+u_{2,x}) \\[0.5ex] \frac12(u_{1,y}+u_{2,x}) & \frac12(u_{2,y}-u_{1,x})
 \end{pmatrix}= 0 \quad \Leftrightarrow \quad \begin{cases}
                                               u_{1,x}&=u_{2,y} \\
                                               u_{1,y}&=-u_{2,x}
                                              \end{cases}
\end{align*}
and corresponds to the validity of the Cauchy-Riemann-equations. Thus, in the planar case, infinitesimal conformal mappings are conformal mappings.

\subsection{Kr\"oner's relation in infinitesimal elasto-plasticity}
At the macroscopic scale, in infinitesimal elasto-plastic theory, see e.g. \cite{agn_ebobisse2018well,agn_ebobisse2017existence,agn_ebobisse2017fourth,Li2008,amstutz:hal-01789190,Amstutz2016analysis,Maggiani2015incompatible}, the incompatibility of the elastic strain is related to the $\Curl$ of the \textit{contortion tensor} $\kappa\coloneqq\alpha^T -\frac12\tr(\alpha)\cdot\id$, where $\alpha\coloneqq \Curl P$ is the dislocation density tensor, by Kr\"oner's relation \cite{Kroener1954}:
\begin{equation}\label{eq:Kroener}
  \inc(\sym e) = -\Curl \kappa,
\end{equation}
where the additive decomposition of the displacement gradient into non-symmetric elastic and plastic distortions is assumed:
\begin{equation}\label{eq:desompKroen1}
 \D u = e + P.
\end{equation}
Indeed, \eqref{eq:Kroener} follows from Nye's formula \eqref{eq:Nye} and the identities
\[
 \tr\Curl\sym e = 0 \quad \text{as well as} \quad \alpha\coloneqq\Curl P\overset{\eqref{eq:desompKroen1}}{=}-\Curl e,
\]
since we have
\begin{align}
 \D \axl \skew e~ &\overset{\mathclap{\eqref{eq:Nye}_2}}{=}\hspace{1ex}\frac12\tr(\Curl\skew e)\cdot\id-(\Curl\skew e)^T\hspace{3em}
 \overset{\mathclap{\tr\Curl\sym e = 0}}{=}\hspace{3em} \frac12\tr(\Curl\skew e+\Curl\sym e)\cdot\id-(\Curl\skew e)^T\notag\\
 &=\hspace{1ex} \frac12\tr(\Curl e)\cdot\id-(\Curl e)^T+ (\Curl\sym e)^T\hspace{2em}
 \overset{\mathclap{\alpha=-\Curl e}}{=} \hspace{2em}~ -\frac12\tr(\alpha)\cdot\id + \alpha^T + (\Curl\sym e)^T\notag\\
 & = \hspace{1ex} \kappa + (\Curl\sym e)^T\label{eq:vorarbeitKroener}.
\end{align}
Thus, applying $\Curl$ on both sides of \eqref{eq:vorarbeitKroener} establishes \eqref{eq:Kroener}, since $\Curl\circ \D \equiv0$:
\begin{equation}
 0 = \Curl \D\axl \skew e \overset{\eqref{eq:vorarbeitKroener}}{=} \Curl \kappa +\Curl([\Curl\sym e]^T)= \Curl \kappa + \inc (\sym e).
\end{equation}
From the decomposition ~ $\sym \D u = \sym e + \sym P$ ~ it follows moreover ~ $\inc(\sym e) = -\inc(\sym P)$, see also last calculation in footnote \ref{footnote:inc}. \medskip

In finite strain elasticity \cite{Ciarlet2005intro}, the Riemann-Christoffel tensor $\mathcal{R}$ expresses the compatibility of strain tensors in the sense of
\begin{equation}
 C\in C^2(\Omega,\Sym^+(3)): \quad \mathcal{R}(C)= 0 \quad \Leftrightarrow \quad C = (D\varphi)^T D\varphi \quad \text{in simply connected domains}.
\end{equation}
Writing ~ $C=(\id+P)^T(\id +P)= 1 + 2\,\sym P + P^TP$ ~ for $P\in C^2(\Omega,\R^{3\times3})$, the incompatibility operator is the linearization of the Riemann-Christoffel tensor at the identity, since
\begin{equation}
\mathcal{R}(\id + 2\,\sym P + P^TP) = \mathcal{R}(\id)+2\D \mathcal{R}(\id)\,\sym P + \text{h.o.t.} = 0 + 2\,\inc(\sym P) + \text{h.o.t.}
\end{equation}
see also \cite{agn_ebobisse2017fourth} and the references contained therein.

\subsection{Further identities}
Symmetric tensors play an important role in the above considerations. We mention here the full expression of $S\times b$ for $S\in\Sym(3)$ and $b\in\R^3$:
\begin{equation}
S \times b = \begin{pmatrix}
              S_{12}\,b_3 - S_{13}\,b_2 & S_{13}\, b_1 - S_{11}\,b_3 & S_{11}\,b_2 - S_{12}\,b_1\\
              S_{22}\,b_3 - S_{23}\,b_2 & S_{23}\, b_1 - S_{12}\,b_3 & S_{12}\,b_2 - S_{22}\,b_1\\
              S_{23}\,b_3 - S_{33}\,b_2 & S_{33}\, b_1 - S_{13}\,b_3 & S_{13}\,b_2 - S_{23}\,b_1
              \end{pmatrix}
\end{equation}
which is an example of a trace-free matrix with non-zero entries on the diagonal:
\[
\tr(S \times b) = S_{12}\,b_3 - S_{13}\,b_2  + S_{23}\, b_1 - S_{12}\,b_3 + S_{13}\,b_2 - S_{23}\,b_1 = 0.
\]
Moreover, we outline some basic identities which played useful roles in our considerations:
\bigskip

\hspace{-4ex}
\begin{tabular}{:l:l:}
\hdashline
&\\
 1. from linear algebra: & 2. and their formal equivalents from calculus:\\[1ex]
 \ (a) \ $P\times b$ \quad row-wise, & \ (a) \ $\Curl P = P \times(-\nabla)$,\\[0.8ex]
 \ (b) \ $\id\times b = \anti(b)\in\so(3)$, & \ (b) \ $\Curl(\zeta\cdot\id)=-\anti(\nabla \zeta)\in\so(3)$,\\[0.8ex]
 \ (c) \ $(\anti a)\times b = b \otimes a -\skalarProd{b}{a}\,\id$,&  \ (c) \ $\Curl A = \tr(\D \axl A)\,\id- (\D \axl A)^T$,\\[0.8ex]
 \ (d) \ $\tr(S\times b)=0$,&\ (d) \ $\tr(\Curl S)=0$,\\[0.8ex]
 \ (e) \ $\big(\id\times b\big)^T\times b = \norm{b}^2\cdot\id - b\otimes b\in\Sym(3)$,&  \ (e) \  $\inc(\zeta\cdot\id)= \Delta\zeta\cdot \id-\D^2\zeta\in\Sym(3)$,\\[0.8ex]
  \ (f) \ $\big((\anti a)\times b\big)^T\times b = -\skalarProd{b}{a}\anti(b)\in\so(3)$,\hspace*{2ex}&  \ (f) \ $\inc A = -\anti(\nabla \tr(\D \axl A))\in\so(3)$,\\[0.8ex]
   \ (g) \ $\big(S\times b\big)^T\times b \in\Sym(3)$, &  \ (g) \ $\inc S \in\Sym(3)$,\\[0.8ex]
    \ (h) \ $\dev(P\times b)= P\times b + \frac23\skalarProd{\axl\skew P}{b}\cdot\id$,&  \ (h) \ $\dev\Curl P = \Curl P -\frac23\div\axl\skew P\cdot \id$,\\[0.8ex]
     \ (i) \  $\sym[(P\times b)^T\times b] = ((\sym P)\times b)^T\times b$, & \ (i) \ $\sym\inc P = \inc\sym P$,\\[0.8ex]
     \ (j) \  $\skew[(P\times b)^T\times b] = ((\skew P)\times b)^T\times b$, & \ (j) \ $\skew\inc P = \inc\skew P$,\\[1.4ex]
   \begin{minipage}{6.5cm}
              \ (k) \ $a\otimes b=0 \quad \Leftrightarrow\quad \sym(a\otimes b)= 0$ \\
             \hphantom{ \ (k) \  $a\otimes b=0 \quad$}$\Leftrightarrow \quad \dev(a\otimes b)=0$\\
             \hphantom{ \ (k) \  $a\otimes b=0 \quad$}$\Leftrightarrow \quad \dev\sym(a\otimes b)=0$,
            \end{minipage}
&  \begin{minipage}{7cm} for ~$\zeta\in\mathscr{D}'(\Omega,\R)$, $A\in\mathscr{D}'(\Omega,\so(3))$,\\ \hphantom{for ~}$S\in\mathscr{D}'(\Omega,\Sym(3))$ and $P\in\mathscr{D}'(\Omega,\R^{3\times 3})$. \end{minipage} \\[4ex]
\ (l) \ $\dev(P\times b) = 0 \quad \Leftrightarrow \quad P\times b = 0$, & \\[2ex]
for $a,b \in\R^3$, $S\in\Sym(3)$  and $P\in\R^{3\times 3}$, &\\
&\\
\hdashline
\end{tabular}
\bigskip

\noindent We catch up with the verification of the identities not contained in our considerations explicitly:
\begin{itemize}
 \item
$ \big(\id\times b\big)^T\times b \overset{\text{1.(a)}}{=}(\anti(b))^T\times b = -(\anti b)\times b \overset{\text{1.(b)}}{=} - b\otimes b +\skalarProd{b}{b}\,\id  \quad \Rightarrow \quad \text{1.(d)}$,
\item we have the decompositions:
\begin{align*}
 (P\times b)^T\times b &= (\sym P \times b + \skew P\times b)^T \times b = \underset{\in\Sym(3)}{\underbrace{((\sym P)\times b)^T\times b}}  + \underset{\in\so(3)}{\underbrace{((\skew P)\times b)^T\times b}}
 \shortintertext{but also}
 \inc P & = \inc(\sym P + \skew P) = \underset{\in\Sym(3)}{\underbrace{\inc\sym P}} + \underset{\in\so(3)}{\underbrace{\inc \skew P}}
 \end{align*}
 where we have used (e) and (f), so that (h) and (i) follow,
\item
the equivalence ~ $a\otimes b = 0$ ~ $\Leftrightarrow$ ~ $\dev\sym(a\otimes b)=0$ follows from the expression:
\begin{equation*}
 \frac{\norm{b}^4}{2}\norm{a\otimes b}^2 = \norm{b}^4\norm{\dev\sym(a\otimes b)}^2+\frac12\left(\frac{n}{n-1}\right)^2\skalarProd{b}{\dev\sym(a\otimes b) b}^2.
\end{equation*}
\end{itemize}
\end{alphasection}
}

\end{document}